\documentclass[twoside]{amsart}

\numberwithin{equation}{section}
\usepackage{amsmath,amssymb,amsfonts,amsthm,latexsym}
\usepackage{graphicx,color,enumerate}
\usepackage[margin=1in]{geometry}
\usepackage{longtable}
\usepackage{appendix}
\usepackage{xcolor}
\usepackage{stmaryrd}
\usepackage{mleftright}
\usepackage{breqn}
\usepackage{url}
\usepackage{orcidlink}

\newtheorem{theorem}{Theorem}[section]
\newtheorem{lemma}[theorem]{Lemma}
\newtheorem{proposition}[theorem]{Proposition}
\newtheorem{corollary}[theorem]{Corollary}

\theoremstyle{definition}
\newtheorem{definition}[theorem]{Definition}

\theoremstyle{remark}

\newcommand{\R}{\mathbb{R}}

\newcommand{\Z}{\mathbb{Z}}

\usepackage{graphicx} % Required for inserting images
\usepackage{tikz,tikz-3dplot}
\tikzset{surface/.style={draw=black, fill=white, fill opacity=.6}}

\begin{document}

\title{The symplectic form associated to a singular Poisson algebra}

\author[H.-C.~Herbig]{Hans-Christian Herbig\,\orcidlink{0000-0003-2676-3340}}
\address{Departamento de Matem\'{a}tica Aplicada, Universidade Federal do Rio de Janeiro,
Av. Athos da Silveira Ramos 149, Centro de Tecnologia - Bloco C, CEP: 21941-909 - Rio de Janeiro, Brazil}
\email{herbighc@gmail.com}

\author[W. Osnayder Clavijo E.]{
William Osnayder Clavijo Esquivel\,\orcidlink{0009-0000-4211-3300
}}
\address{Departamento de Matem\'{a}tica Aplicada, Universidade Federal do Rio de Janeiro,
Av. Athos da Silveira Ramos 149, Centro de Tecnologia - Bloco C, CEP: 21941-909 - Rio de Janeiro, Brazil}
\email{woclavijo@im.ufrj.br}

\author[C.~Seaton]{Christopher Seaton\,\orcidlink{ 0000-0001-8545-2599}}
\address{Skidmore College
815 North Broadway
Saratoga Springs, NY 12866, USA}
\email{cseaton@skidmore.edu}


\begin{abstract}
Given an affine Poisson algebra, that is singular one may ask whether there is an associated symplectic form. In the smooth case the answer is obvious: for the symplectic form to exist the Poisson tensor has to be invertible. In the singular case, however, derivations do not form a projective module and the nondegeneracy condition is more subtle. For a symplectic singularity one may naively ask if there is indeed an analogue of a symplectic form. We examine an example of a symplectic singularity, namely the double cone, and show that here such a symplectic form exists. We use the naive de Rham complex of a Lie-Rinehart algebra. Our analysis of the double cone uses Gröbner bases calculations. We also give an alternative construction of the symplectic form that generalizes to categorical quotients of cotangent lifted representations of finite groups. We use the same formulas to construct a symplectic form on the simple cone, seen as a Poisson differential space and generalize the construction to linear symplectic orbifolds. We present useful auxiliary results that enable to explicitly determine generators for the module of derivations an affine variety. The latter may be understood as a differential space.
\end{abstract}

\maketitle
\tableofcontents
\section{Introduction}\label{sec:intro}

Many geometric structures from differential geometry are defined in terms of tensors, i.e., in terms of multilinear algebra of tangent and cotangent bundles. In the presence of singularities, however, the modules of sections of those `canonical bundles' do not form projective modules. For this reason it is often not easy to generalize those geometric structures to singular situations. A noteworthy exception is Poisson geometry, since it is defined in terms of functions. The geometry of singular Poisson structures has been studied for decades in the context of symplectic reduction \cite{AGJ,SL,SniatyckiBook,HSScomp}, coadjoint orbit closures \cite{BaohuaFu} and moduli spaces \cite{Huebsch, HSSrat}. Many of those singular Poisson structures exhibit `symplectic' features. For example, symplectic reductions carry symplectic stratifications \cite{SL} and have typically symplectic singularities \cite{HSScomp}. However, we do not know of any attempt to develop a theory of symplectic forms for singular affine Poisson varieties. In this paper we suggest such an approach and show that it makes sense in the case of orbifolds. This is done in the framework of complex algebraic geometry and in the framework of Poisson differential spaces. We leave it to the future to elaborate our idea for more complicated geometries, such as symplectic reductions, coadjoint orbit closures, symplectic singularities and gauge theoretic moduli spaces (for more details see Section \ref{sec:out}).

There have been suggestions to generalize differential forms to singular spaces, see, e.g., \cite{Knighten, Huber, Kebekus} in the context of complex algebraic geometry or \cite{Watts,axioms} in the context of differential spaces.\footnote{We do not attempt here to elaborate here on the history of the subject.} We use the language of Lie-Rinehart algebras as a unifying principle and base our considerations on what we call the \textit{naive de Rham complex} (see Section \ref{sec:naive}), which originates from George Rinehart's work \cite{Rinehart}. In complex algebraic geometry this is referred to as the de Rham complex of \emph{reflexive differential forms} \cite{Huber, Kebekus} or \emph{Zarisky differentials} \cite{Knighten}. The difficulties in working with differential forms in the presence of singularities are due to the fact that modules of differentials have torsion and are not reflexive. In particular, the dual module to the module of derivations of a singular algebra is the reflexive closure of the module of differentials.

It should be also said that algebraic geometry is not the proper setting for symplectic geometry since Hamiltonian flows do not respect polynomial observables.
This problem could be cured by moving to the holomorphic framework, however, this is not the most natural (i.e., physically adequate) Ansatz. The appropriate framework for singular symplectic geometry appears to be that of a Poisson differential space $\left(\mathcal X,\mathcal C^\infty (\mathcal X) ,\{\ ,\  \} \right)$ in the sense of \cite{FHS,SniatyckiBook}, or variations thereof such as, e.g., \cite{SL,stratKaeh}. This is because symplectic reductions by compact group actions and gauge theoretic moduli spaces are to be described in this language. The notion of the module of differentials $\mathcal D_{\mathcal C^\infty(\mathcal X)|\mathbb R}$ for a differential space has been developed in \cite{Navarro} and the idea of using the naive de Rham complexes of the Lie-Rinehart algebra $(\mathcal C^\infty( X) ,\operatorname{Der}\left(\mathcal C^\infty (\mathcal X)\right)$ goes through without difficulty.

Let us fix some notations and conventions. Throughout this article $\boldsymbol{k}$ denotes a field of characteristic zero. If $A$ is a commutative $\boldsymbol{k}$-algebra we denote by
$ \operatorname{Der}( A) =\left\{X:A\rightarrow A\ |\ \forall a,b\in A: \ \ X( ab) =X( a) b+aX( b)\right\}$ the $A$-module of $A$-valued derivations. We consider a polynomial $\boldsymbol{k}$-algebra $P=\boldsymbol{k}\left[ x^{1} ,\dotsc ,x^{n}\right]$ with symplectic Poisson bracket $\{\ ,\ \}$. Let us write for the associated Poisson tensor $\Pi ^{ij} =\left\{x^{i} ,x^{j}\right\} \in P$. Let $I=( f_{1} ,\dotsc ,f_{\ell })$ be a Poisson ideal in $P$, i.e., a multiplicative ideal such that $\ \{I,P\} \subseteq I$ (for examples see \cite{higherKoszul}). Its generators $\ f_{\mu }$ have the property that $\left\{x^{i} ,f_{\mu }\right\} =\sum _{\nu =1}^{\ell } Z_{\mu }^{i\nu } f_{\nu }$ for some (in general, non-unique) $Z_{\mu }^{i\nu } \in P$. As the $Z_{\mu }^{i\nu }$ can be interpreted as Christoffel symbols of a connection (see \cite{Nambuffel,higherKoszul}) of the conormal module $I/I^{2}$ we refer to them as the \textit{Poissoffel symbols} of the Poisson ideal. The quotient $A=P/I$ becomes a Poisson $\boldsymbol{k}$-algebra. We refer to this type of algebra as an \textit{affine Poisson algebra}. We apologize to the algebraic readership that we adopt the physicist's habit to use upper indices.

Let us outline the plan of the paper. In Section \ref{sec:naive} we recall the naive de Rham complex of a Lie-Rinehart algebra and define what we mean by a non-degenerate $2$-form. In Section \ref{sec:main} we explain our definition of a symplectic form, generalizing the usual definition to the singular case. In Section \ref{sec:dc} we find such a symplectic form on the double cone. In Section \ref{sec:expl}
we explain the findings of Section \ref{sec:dc} in terms of orbifold symplectic forms. In Section \ref{sec:Diff} we develop a version of Section \ref{sec:dc} for the simple cone in the framework of Poisson differential spaces. In Section \ref{sec:orbifold} we construct symplectic forms on linear symplectic orbifolds, seen as Poisson differential spaces. In the final Section \ref{sec:out} we suggest a plan for further, more profound, investigations.

\vspace{2mm}

\noindent \emph{Acknowledgements.} HCH would like to thank Ines Kath for helping him think and Henrique Bursztyn sending a copy of \cite{BierstoneIMPA}. We are deeply indepted to Gerald Schwarz for his vision and encouragement.

\section{Lie-Rinehart algebras and the naive de Rham complex}\label{sec:naive}

In order to be able to explain our main idea we need to recall some basic material about the cohomology of Lie-Rinehart algebras (cf. \cite{Rinehart}).

 A \textit{Lie-Rinehart algebra} $ ( L,A)$ is a commutative $ \boldsymbol{k}$-algebra $ A$ and an $ A$-module $ L$ such that
\begin{enumerate}
    \item
 $ L$ is a $ \boldsymbol{k}$-Lie algebra,
\item  $ L$ acts on $ A$ by derivations via $ L\otimes _{\boldsymbol{k}} A\rightarrow A,\ X\otimes a\mapsto X( a)$,
\item  $ [ X,aY] =X( a) Y+a[ X,Y]$ for all $ X,Y\in L$ and $ a\in A$,
\item  $ ( aX)( b) =aX( b)$ for all $ X\in L$ and $ a,b\in A$.
\end{enumerate}
The $ A$-module of $ n$-cochains of the \textit{naive de Rham complex} of the Lie-Rinehart algebra $( L,A)$ is given by the space $ \operatorname{Alt}^{n}_A( L,A)$ of alternating $A$-linear forms of arity $n$ with values in $A$. The differential $ \operatorname{d} :\operatorname{Alt}^{n}_A( L,A)\rightarrow \operatorname{Alt}_A^{n+1}( L,A)$ of the naive de Rham complex is given by the \emph{Koszul formula}
\begin{align*}
\left(\operatorname{d} \omega \right)( X_{0} ,X_{1} ,\dotsc ,X_{n}) &=\sum _{i=0}^{n}( -1)^{i} X_{i} \left(\omega \left( X_{0} ,\dotsc ,\widehat{X_{i}} ,\dotsc ,X_{n}\right)\right)\\
& \ \ \ \ \ \ \ \ \ +\sum _{i< j}^{n}( -1)^{i+j}\omega\left([X_i,X_j],( X_{0} ,\dotsc ,\widehat{X_{i}} ,\dotsc ,\widehat{X_{j}} ,\dotsc ,X_{n}\right) ,
\end{align*}
where $ X_{0} ,X_{1} ,\dotsc ,X_{n} \in L$ and $ \operatorname{Alt}^{n}_A( L,A)$. The $ \widehat{\ \ }$ indicates omission of the corresponding term.
It is well-known (see \cite{Palais}) that $ \left(\operatorname{Alt}^{n}_A( L,A) ,\operatorname{d}\right)$ forms a differential graded $ \boldsymbol{k}$-algebra with respect to the product $ \cup :\operatorname{Alt}^{p}_A( L,A) \times \operatorname{Alt}^{q}_A( L,A)\rightarrow \operatorname{Alt}^{p+q}_A( L,A)$
\begin{align*}
 \omega \cup \eta ( X_{1} ,X_{2} ,\dotsc ,X_{p+q}) =\sum _{\sigma \in \operatorname{Sh}_{p,q}}( -1)^{\sigma } \omega ( X_{\sigma ( 1)} ,\dotsc ,X_{\sigma ( q)}) \eta ( X_{\sigma ( p+1)} ,\dotsc ,X_{\sigma ( p+q)})
\end{align*}
where $ \operatorname{Sh}_{p,q}$ denotes the set of $ p,q$-shuffle permutations. If $ L=\operatorname{Der}( A)$ we use the notation $ \operatorname{d_{dR}} :=\operatorname{d}$. Rinehart \cite{Rinehart} has shown that if $ L$ is a projective $ A$-module then the $ n$th cohomology of $ \left(\operatorname{Alt}^{n}_A( L,A) ,\operatorname{d}\right)$ computes $ \operatorname{Ext}_{U( L,A)}^{n}( A,A)$, where $ U( L,A)$ denotes Rinehart's universal enveloping algebra of $ ( L,A)$. If $ L$ is not projective there is no such result to be expected in general. For this reason the complex is called naive.

For affine $ \mathbb{C}$-algebras $ A=\mathbb{C}\left[ x^{1} ,\dotsc ,x^{n}\right] /I$ the naive de Rham complex of $ \left(\operatorname{Der}( A) ,A\right)$ appears to be isomorphic to the complex of \textit{reflexive differential forms}, see, e.g., \cite{Huber}.

We say that $ \omega \in \operatorname{Alt}_{A}^{2}( L,A)$ is \textit{non-degenerate} if for each $ X\in L$ we have that $ \omega ( X,\ ) =0$, seen as an element in $ \operatorname{Hom}_{A}( L,A) =\operatorname{Alt}_{A}^{1}( L,A)$, implies $ X=0$. Note that if $ A=\boldsymbol{k}$ and $ \dim_{\boldsymbol{k}} L< \infty $ for a non-degenerate $ \omega \in \operatorname{Alt}_{\boldsymbol{k}}^{2}( L,\boldsymbol{k})$ then $ L\rightarrow L^{*} =\operatorname{Hom}_{\boldsymbol{k}}( L,\boldsymbol{k}) ,\ X\mapsto \omega ( X,\ )$ is injective and hence onto. However, the $ A$-linear map $ L\rightarrow L^{\lor } :=\operatorname{Hom}_{A}( L,A) ,\ X\mapsto \omega ( X,\ )$ is not onto in general, even if $ \omega $ is non-degenerate. We write $ \ker \omega =\{X\in L\ |\ \forall Y\in L:\ \omega ( X,Y) =0\}$ for the \textit{kernel} of $ \omega $.

\section{The main idea}\label{sec:main}
As a warm-up and to motivate our approach we recall how to invert a non-degenerate constant Poisson tensor on $\boldsymbol{k}^n$. Assume that $ \Pi ^{ij} =\left\{x^{i} ,x^{j}\right\}\in\boldsymbol{k} $ is a non-degenerate Poisson tensor of the Poisson algebra $(P=\boldsymbol{k}[x^1,\dots,x^n],\{\ , \ \})$. Then
$ \left(\operatorname{d} x^{i}\right)^{\sharp } =\sum _{j} \Pi ^{ij} \partial _{j}$ defines the \emph{musical isomorphism} $\sharp:\Omega_{P|\boldsymbol{k}}\to \operatorname{Der}(P)$ from the $P$-module of Kähler differentials $\Omega_{P|\boldsymbol{k}}$ to the $P$-module vector fields $\operatorname{Der}(P)$.
Then the symplectic form $\omega=\sum_{i,j}\frac{1}{2}\omega_{ij}\operatorname{d}x^i\wedge\operatorname{d} x^j$ is defined by
$$\omega \left(\left(\operatorname{d} x^{i}\right)^{\sharp } ,\partial _{k}\right) =\sum _{j} \omega \left( \Pi ^{ij} \partial _{j} ,\partial _{k}\right) =\sum _{j} \Pi ^{ij} \omega _{jk} =\delta _{k}^{i} =\partial_k x^{i}.$$
This can be rewritten in a coordinate free way as follows
$ \omega \left(\left(\operatorname{d} a\right)^{\sharp } ,X\right) =\sum _{i,j,k}\frac{\partial a}{\partial x^{i}} \omega \left( \Pi ^{ij} \partial _{j} ,X^{k} \partial _{k}\right) =X( a)$ for $X\in\operatorname{Der}(P)$ and $a\in P$.

In order to discuss the singular case let us recall the following.
\begin{theorem}\cite[Lemma 2.1.2] {Simis}
    If $ I\subseteq P$ be an ideal in a polynomial $\boldsymbol{k}$-algebra $P$ then $ \operatorname{Der}( A) \simeq \operatorname{Der}_{I}( P) /I\operatorname{Der}( P)$, where
    $\operatorname{Der}_{I}( P)=\{X\in \operatorname{Der}(P)| X(I)\subseteq I\}$.
\end{theorem}

With this isomorphism understood we define a version of the musical map in the singular case as follows. It is given as the $A$-linear map
$$ \sharp :\Omega _{A|\boldsymbol{k}}\rightarrow \operatorname{Der}( A), \ \ \left(\operatorname{d}( a+I)\right)^{\sharp } =\{a,\ \} +I\operatorname{Der}_{I}( P),$$
where $\Omega _{A|\boldsymbol{k}}$ denotes the $A$-module of Kähler differentials and $a\in P$. Its image is denoted by $(\Omega _{A|\boldsymbol{k}})^\sharp$. As we will see in the next section $(\Omega _{A|\boldsymbol{k}})^\sharp$ may be different from $\operatorname{Der}( A)$. We denote by $ \iota ^{\operatorname{Ham}} :(\Omega _{A|\boldsymbol{k}})^\sharp\rightarrow \operatorname{Der}( A)$ the inclusion.

\begin{lemma}
 The map $\sharp$ is a morphism of Lie-Rinehart algebras and, accordingly, its image $(\Omega _{A|\boldsymbol{k}})^\sharp$ a Lie-Rinehart subalgebra of $(\operatorname{Der}(A),A)$.
\end{lemma}
\begin{proof}
    Consider Kähler forms $(a_1+I)\operatorname{d}(a_2+I)$ and $(b_1+I)\operatorname{d}(b_2+I)$ for $a_1,a_2,b_1,b_2\in P$.
    Recall \cite{Huebsch} that $(\Omega _{A|\boldsymbol{k}},A)$ forms a Lie-Rinehart algebra whose bracket is the so-called Koszul bracket:
    \begin{align*}
&[(a_{1} +I)\operatorname{d} (a_{2} +I),(b_{1} +I)\operatorname{d} (b_{2} +I)]\\
&=(a_{1} b_{1} +I)\operatorname{d} (\{a_{2} ,b_{2}\} +I)+(a_{1}\{a_{2} ,b_{1}\} +I)\operatorname{d} (b_{2} +I)-(b_{1}\{b_{2} ,a_{1}\} +I)\operatorname{d} (a_{2} +I).
    \end{align*}
    On the other hand we have the commutator
    \begin{align*}
        [ a_{1}\{a_{2} ,\ \} +I,b_{1}\{b_{2} ,\ \} +I] =a_{1} b_{1}\{\{a_{2} ,b_{2}\} ,\ \} +a_{1}\{a_{2} ,b_{1}\}\{b_{2} ,\ \} -b_{1}\{b_{2} ,a_{1}\}\{a_{2} ,\ \} +I.
    \end{align*}
\end{proof}

\begin{definition}\label{def:DHam}
We define the $A$-bilinear map $ \omega ^{\operatorname{Ham}} :(\Omega _{A|\boldsymbol{k}})^\sharp \times\operatorname{Der}( A)\to A$ by
\begin{align*}
\omega ^{\operatorname{Ham}}\left(\left(\operatorname{d} a+I\right)^{\sharp } ,X+I\operatorname{Der}_{I}( P)\right) :=X( a) +I.
\end{align*}
\end{definition}

\begin{proposition} $\ $
\begin{enumerate}
\item The form $ \omega ^{\operatorname{Ham}}$ in Definition \ref{def:DHam} does not depend on the choice of the representatives $a\in P$, $X\in \operatorname{Der}_I(P)$ and is hence well-defined.
\item With $ \delta ^{\operatorname{Ham}} :\operatorname{Alt}^{n}_A\left( (\Omega _{A|\boldsymbol{k}})^\sharp ,A\right)\rightarrow \operatorname{Alt}^{n+1}_A\left( (\Omega _{A|\boldsymbol{k}})^\sharp ,A\right)$ the naive de Rham differential of the Lie-Rinehart algebra $\left( (\Omega _{A|\boldsymbol{k}})^\sharp ,A\right)$
the restriction of $ \omega ^{\operatorname{Ham}}$ to $ \operatorname{Alt}^{2}_A\left( (\Omega _{A|\boldsymbol{k}})^\sharp ,A\right)$ fulfills $ \delta ^{\operatorname{Ham}} \omega^{\operatorname{Ham}} =0$.
\end{enumerate}
\end{proposition}
\begin{proof}
   % \colorbox{red}{William!}
   % (1) argumemtação não é sólio
%To prove $(1)$, we Consider $X, Y \in \operatorname{Der}(P)$,$ a, b \in P$ such that $a+I=b+I$ and $X+I\operatorname{Der}_{I}( P) = Y+I\operatorname{Der}_{I}( P)$ then
%\begin{equation*}
% X( a) +I =\omega ^{\operatorname{Ham}}\left(\left(\operatorname{d} a+I\right)^{\sharp } ,X+I\operatorname{Der}_{I}( P)\right)  = \omega ^{\operatorname{Ham}}\left(\left(\operatorname{d} b+I\right)^{\sharp } ,Y+I\operatorname{Der}_{I}( P)\right) =Y( b) +I
%\end{equation*}
As (1) is clear we address (2).
Consider $a, b, c \in P$ with $X= \{a,\ \} +I\operatorname{Der}_{I}( P)$, $Y=\{b,\ \} +I\operatorname{Der}_{I}( P)$, $Z=\{c,\ \} +I\operatorname{Der}_{I}( P)$
\begin{eqnarray*}
\delta ^{\operatorname{Ham}} \omega^{\operatorname{Ham}} (X,Y,Z) &=& X(\omega ( Y,Z)) -Y(\omega ( X,Z)) +Z(\omega ( X,Y))-\omega ([ X,Y] ,Z) +\omega ([ X,Z] ,Y) -\omega ([ Y,Z] ,X), \\
&=& \{ a, \{c , b\} \} - \{ b, \{c , a\} \} + \{ c, \{b , a\} \} -\{ a, \{c , b\} \} + \{ b, \{a , c \} -\{ a, \{b , c\} \} \}+I \in I
\end{eqnarray*}
by Jacobi's identity.
\end{proof}

We are now in the position to formulate the \textbf{fundamental question} of our approach. Does there exist a non-degenerate $ \omega \in \operatorname{Alt}^{2}_A\left(\operatorname{Der}( A) ,A\right)$ such that
\begin{enumerate}
    \item $ \operatorname{d}_{\operatorname{dR}} \omega =0$ and
\item  $ \omega \left( \iota ^{\operatorname{Ham}} \otimes \operatorname{id}\right) =\omega ^{\operatorname{Ham}}$\ ?
\end{enumerate}
In this case we say that $\omega$ is a \emph{symplectic form} on $\operatorname{Spec}(A)$. We will see in the next section that the answer can be affirmative. A priori, such a symplectic form does not have to be unique. The question of unicity will not be addressed in this paper.

\section{The double cone}\label{sec:dc}
Let us check if the program laid out in the Section \ref{sec:main} makes sense for the double cone
$$A=\boldsymbol{k}\left[ x^{1} ,x^{2} ,x^{3}\right] /\left( x^{1} x^{2} -x^{3} x^{3}\right) =:P/( f).$$
\begin{figure}
    \centering
    \includegraphics[width=0.4\textwidth]{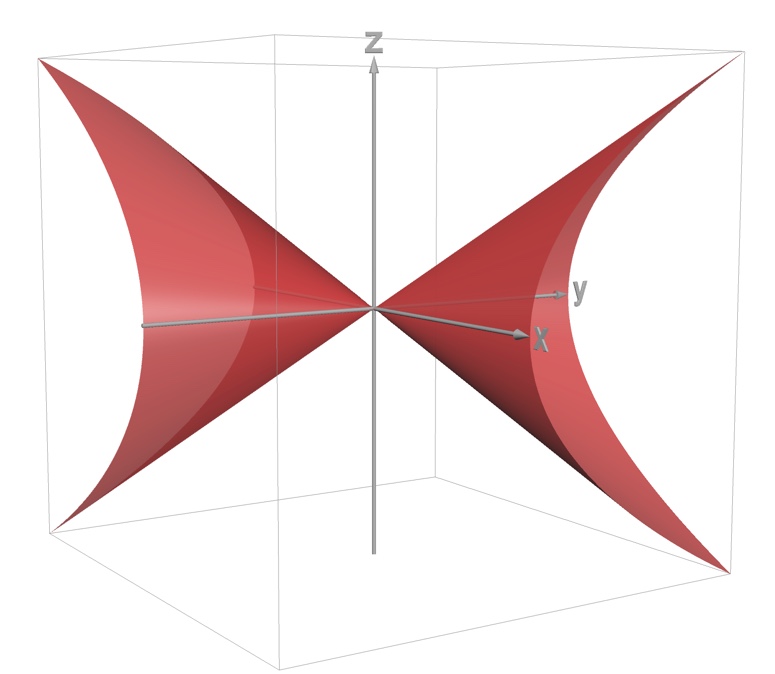}
    \caption{The double cone.}
    \label{fig:dc}
\end{figure}

%\begin{figure}
%    \centering
%    \begin{tikzpicture}[tdplot_main_coords]
%  \coordinate (O) at (0,0,0);

  %% make sure to draw everything from back to front
%  \coneback[surface]{-3}{2}{-10}
%  \draw (0,0,-5) -- (O);
%  \conefront[surface,fill=orange,opacity=5]{-3}{2}{-10}
 % \draw[->] (-6,0,0) -- (6,0,0) node[right] {$x$};
 % \draw[->] (0,-6,0) -- (0,6,0) node[right] {$y$};
 % \coneback[surface]{3}{2}{10}
 % \draw[->] (O) -- (0,0,5) node[above] {$z$};
 % \conefront[surface,fill=green,opacity=5]{3}{2}{10}
%\end{tikzpicture}
  %  \caption{The double cone}
  %  \label{fig:enter-label}
%\end{figure}

%\begin{figure}
%    \centering
%    \begin{tikzpicture}%https://tex.stackexchange.com/a/28775/
%  \begin{axis}[title= Double cone, domain=0:5, y domain=0:2*pi,xmin=-10, xmax=10, ymin=-10, ymax=10, samples=20]
 % \addplot3 [surf,z buffer=sort] ({x*cos(deg(y))}, {x*sin(deg(y))}, {x});
%  \addplot3 [surf,z buffer=sort] ({x*cos(deg(y))}, {x*sin(deg(y))}, {-x});
 % \end{axis}
 % \end{tikzpicture}
  %  \caption{The double cone}
   % \label{fig:enter-label}
%\end{figure}

We view $ A=\boldsymbol{k}\left[ x^{1} ,x^{2} ,x^{3}\right] /\left( x^{1} x^{2} -x^{3} x^{3}\right)$ as the Poisson algebra of invariants of the linear cotangent lifted $ \mathbb{Z}_{2}=\operatorname{O}_2(\boldsymbol{k})$-action on $ \boldsymbol{k}^{2}=T^*\boldsymbol{k}$ (see \cite{higherKoszul}).
Here the coordinates $x^{1} ,x^{2}$ and $x^{3}$ correspond to the $\Z_2$-invariants $qq,pp$ and $qp$ understood with canonical bracket $\{q,p\}=1$.
The variety $\operatorname{Spec}(A)$ has an isolated symplectic singularity at the origin (see \cite{HSScomp}).

\begin{proposition}\label{prop:DerGen}
Let $I=(f_{1} ,\dotsc ,f_{k} )\subset P$ be a ideal and $A=P/I$. Consider the intersection $J_{r} :=\operatorname{Jac}_{f_{r}} \cap I$ of the Jacobian ideal $\operatorname{Jac}_{f_r} =\left(\frac{\partial f_r}{\partial x^{1}} ,\dotsc ,\frac{\partial f_r}{\partial x^{n}}\right) \subseteq P$, $r=1,\dots,k$, with $I$ and the matrix
\begin{align}\label{eq:M}
\begin{bmatrix}
f_{1} & f_{2} & \cdots  & f_{k} & \frac{\partial f_{1}}{\partial x^{1}} & \cdots  & \frac{\partial f_{1}}{\partial x^{n}}\\
f_{1} & f_{2} & \cdots  & f_{k} & \frac{\partial f_{2}}{\partial x^{1}} & \cdots  & \frac{\partial f_{2}}{\partial x^{n}}\\
\cdots  &  &  &  &  &  & \cdots \\
f_{1} & f_{2} & \cdots  & f_{k} & \frac{\partial f_{k}}{\partial x^{1}} & \cdots  & \frac{\partial f_{k}}{\partial x^{n}}
\end{bmatrix} \in P^{k\times ( n+k)}
\end{align}
seen as a $P$-linear morphism $M:P^{n+k}\rightarrow P^{k}$. Then every $X\in \operatorname{Der}_{I}( P)$ can be written as $X=\sum _{j=1}^{n} g^{k+j} \partial /\partial x^{j}$ for some
$$\begin{bmatrix}
g^{1} & g^{2} & \cdots  & g^{k} & g^{k+1} & g^{k+2} & \cdots  & g^{k+n}
\end{bmatrix}^{\top } \in \ker M.$$
In particular, the $P$-module $\operatorname{Ann}_{\boldsymbol{f}}( P) :=\left\{X\in \operatorname{Der}( P) \ |\ \forall r\in\{1,\dots,k\}:\ X( f_r) =0\ \right\}$ is $P$-linearly isomorphic to the intersection $\cap _{r=1}^{k}\operatorname{Syz}\left(\operatorname{Jac}_{f_{r}}\right)$ of the $P$-modules of syzygies $\operatorname{Syz}\left(\operatorname{Jac}_{f_{r}}\right)$ of
the Jacobian ideals $\operatorname{Jac}_{f_{r}}$. Any $f\in \cap _{r=1}^{k} J_{r}$ can be written as $f=-\sum _{j=1}^{k} g^{j} f_{j} =\sum _{j=1}^{n} g^{k+j} \partial f_{r} /\partial x^{j}$ for every $r=1,\dots,k$.
\end{proposition}
\begin{proof}
It is straight forward to verify the statements.
\end{proof}

The table of Poisson brackets is
$$ \begin{array}{ c||c| c |c }
\{\ ,\  \} & x^{1} & x^{2} & x^{3}\\
\hline\hline
x^{1} & 0 & 4x^{3} & 2x^{1}\\\hline
x^{2} & -4x^3 & 0 & -2x^{2}\\\hline
x^{3} & -2x^1 & 2x^2 & 0
\end{array}$$
and the generators of $ (\Omega _{A|\boldsymbol{k}})^\sharp$ as an $A$-module are given by the $I\operatorname{Der}(P)$ classes of
\begin{align}\label{eq:genDHam}
\left\{x^{1} ,\ \right\} =4x^{3}\frac{\partial }{\partial x^{2}} +2x^{1}\frac{\partial }{\partial x^{3}} ,\ \left\{x^{2} ,\ \right\} =-4x^{3}\frac{\partial }{\partial x^{1}} -2x^{2}\frac{\partial }{\partial x^{3}},\ \left\{x^{3} ,\ \right\} =2x^{2}\frac{\partial }{\partial x^{2}} -2x^{1}\frac{\partial }{\partial x^{1}}.
\end{align}
It turns out that the Poissoffel symbols vanish, i.e., the polynomial $ f=x^{1} x^{2} -x^{3}x^{3}$ is actually a Casimir. This means that $ \operatorname{Ann}_{f}(P) =(\Omega _{A|\boldsymbol{k}})^\sharp$.
We have $ \operatorname{Jac}_{f} \cap ( f) =( f)$ since $2f=2x^{1}\partial f/\partial x^{1} +x^{3}\partial f/\partial x^{3}$. So the list \eqref{eq:genDHam} has to be amended by
\begin{align}\label{eq:genEuler}
X_4:=2x^{1}\frac{\partial }{\partial x^{1}} +x^{3}\frac{\partial }{\partial x^{3}}
\end{align}
to get the generators of $ \operatorname{Der}_{I}( P)$ as a $ P$-module. Note that $ 2x^{2}\frac{\partial }{\partial x^{2}} +x^{3}\frac{\partial }{\partial x^{3}} =X_4 - \left\{x^{3} ,\ \right\}$.
According to Macaulay2, the minimal free resolution of the  $A$-module $\operatorname{Der}( A)=\operatorname{coker}\left[\begin{smallmatrix}
0 & -4x^{3} & -2x^{1} & 2x^{1}\\
4x^{3} & 0 & 2x^{2} & 0\\
2x^{1} & -2x^{2} & 0 & x^{3}
\end{smallmatrix}\right]$ starts with
\begin{align}\label{eq:res}
A^{4}\xleftarrow{\left[
\begin{smallmatrix}
x^{2} & 0 & 0 & -x^{3}\\
0 & x^{1} & x^{3} & 0\\
-2x^{3} & 0 & 0 & 2x^{1}\\
-2x^{3} & 2x^{3} & 2x^{2} & 2x^{1}
\end{smallmatrix}\right]} A^4\xleftarrow{\left[
\begin{smallmatrix}
x^{1} & 0 & 0 & x^{3}\\
0 & x^{2} & -x^{3} & 0\\
0 & -x^{3} & x^{1} & 0\\
x^{3} & 0 & 0 & x^{2}
\end{smallmatrix}\right]} A^4\xleftarrow{\left[
\begin{smallmatrix}
x^{2} & 0 & 0 & -x^{3}\\
0 & x^{1} & x^{3} & 0\\
0 & x^{3} & x^{2} & 0\\
-x^{3} & 0 & 0 & x^{1}
\end{smallmatrix}\right]} A^4\leftarrow\dots .
\end{align}
The empirical fact that the differentials are linear seems to indicate that $\operatorname{Der}(A)$ is a Koszul $A$-module. The dimension of $\operatorname{Der}(A)$ is $2$.

%According to Macaulay2, there is a short exact sequence of $A$-modules
%\begin{align}\label{eq:res}
%0\leftarrow \operatorname{Der}( A)\xleftarrow{\left[\begin{matrix}
%0 & -4x^{3} & -2x^{1} & 2x^{1}\\
%4x^{3} & 0 & 2x^{2} & 0\\
%2x^{1} & -2x^{2} & 0 & x^{3}
%\end{matrix}\right]} A^{4}\xleftarrow{\left[\begin{matrix}
%x^{2}\\
%x^{1}\\
%-2x^{3}\\
%0
%\end{matrix}\right]} A\leftarrow 0,
%\end{align}
%so that $\operatorname{pd}(\operatorname{Der}( A))=1$.

Since $ \dim\left(\operatorname{coker}\left( \iota ^{\operatorname{Ham}}\right)\right) =1$ the form $ \omega \in \operatorname{Alt}^{2}_A\left(\operatorname{Der}( A) ,A\right)$ is already defined by Definition \ref{def:DHam}. It remains to check that $ \omega $ is non-degenerate and that $ \operatorname{d}_{\operatorname{dR}} \omega =0$.
To this end let us evaluate
\begin{align*}
\operatorname{d}_{\operatorname{dR}} \omega ( X,Y,X_4) &=X(\omega ( Y,X_4)) -Y(\omega ( X,X_4)) +X_4(\omega ( X,Y))-\omega ([ X,Y] ,X_4) +\omega ([ X,X_4] ,Y) -\omega ([ Y,X_4] ,X),
\end{align*}
where $ X,Y$ are distinct $ \left\{x^{i} ,\ \right\}$ with $ i=1,2,3$.
In fact, we have
\begin{align*}
&\left(\operatorname{d}_{\operatorname{dR}} \omega \right)\left(\left\{x^{1} ,\ \right\} ,\left\{x^{2} ,\ \right\} ,X_4\right)\\&=\left\{x^{1} ,\omega \left(\left\{x^{2} ,\ \right\} ,X_4\right)\right\} -\left\{x^{2} ,\omega \left(\left\{x^{1} ,\ \right\} ,X_4\right)\right\} +X_4\omega \left(\left\{x^{1} ,\ \right\} ,\left\{x^{2} ,\ \right\}\right)\\&\ \ \ -\omega \left(\left[\left\{x^{1} ,\ \right\} ,\left\{x^{2} ,\ \right\}\right] ,X_4\right) +\omega \left(\left[\left\{x^{1} ,\ \right\} ,X_4\right] ,\left\{x^{2} ,\ \right\}\right) -\omega \left(\left[\left\{x^{2} ,\ \right\} ,X_4\right] ,\left\{x^{1} ,\ \right\}\right)\\
&=\left\{x^{1} ,X_4\left( x^{2}\right)\right\} -\left\{x^{2} ,X_4\left( x^{1}\right)\right\} +X_4\left(\left\{x^{2} ,x^{1}\right\}\right)\\
&\ \ \ -\omega \left(\left\{\left\{x^{1} ,x^{2}\right\} ,\ \right\} ,X_4\right) -\omega \left( \{x^1,\ \} ,\left\{x^{2} ,\ \right\}\right)-\omega \left(\left\{x^{2},\ \right\} ,\left\{x^{1} ,\ \right\}\right)\\
&=0-2\left\{x^{2} ,x^{1}\right\} -X_4\left( 4x^{3}\right) -\omega \left(\left\{4x^{3} ,\ \right\} ,X_4\right) +0=8x^{3} -4x^{3} -4x^{3}=0.
\end{align*}
All expressions above are to be understood modulo $I\operatorname{Der}_I(P)$ and $I$, respectively. To unclutter the notation we did not annotate these expressions and continue with this habit later on.
We used the auxiliary evaluations:
\begin{align*}
    \left[\left\{x^{1} ,\ \right\} ,\left\{x^{2} ,\ \right\}\right] &=\left\{x^{1} ,\left\{x^{2} ,\ \right\}\right\} -\left\{x^{2} ,\left\{x^{1} ,\ \right\}\right\} =\left\{\left\{x^{1} ,x^{2}\right\} ,\ \right\},\ \ X_4\left( x^{2}\right) =0, \ \ X_4\left( x^{1}\right) =2x^{1},\\
\left[\left\{x^{1} ,\ \right\} ,X_4\right] &=\left[ 4x^{3}\frac{\partial }{\partial x^{2}} +2x^{1}\frac{\partial }{\partial x^{3}} ,2x^{1}\frac{\partial }{\partial x^{1}} +x^{3}\frac{\partial }{\partial x^{3}}\right] =-2x^{1}\frac{\partial }{\partial x^{3}} -4x^{3}\frac{\partial }{\partial x^{2}} =-\{x^1, \ \},\\
\left[\left\{x^{2} ,\ \right\} ,X_4\right]&=\left[ 4x^{3}\frac{\partial }{\partial x^{2}} +2x^{1}\frac{\partial }{\partial x^{3}} ,2x^{1}\frac{\partial }{\partial x^{1}} +x^{3}\frac{\partial }{\partial x^{3}}\right] =-4x^{3}\frac{\partial }{\partial x^{1}} -2x^{2}\frac{\partial }{\partial x^{3}}=\{x^2, \  \}
\end{align*}
Next we determine
\begin{align*}
&\left(\operatorname{d}_{\operatorname{dR}} \omega \right)\left(\left\{x^{1} ,\ \right\} ,\left\{x^{3} ,\ \right\} ,X_4\right)\\
&=\left\{x^{1} ,\omega \left(\left\{x^{3} ,\ \right\} ,X_4\right)\right\} -\left\{x^{3} ,\omega \left(\left\{x^{1} ,\ \right\} ,X_4\right)\right\} +X_4\omega \left(\left\{x^{1} ,\ \right\} ,\left\{x^{3} ,\ \right\}\right)\\
&\ \ \ \ \ \ \ \ \ \ \ \ \ \ \ \ \ \ \ \ \ \ -\omega \left(\left[\left\{x^{1} ,\ \right\} ,\left\{x^{3} ,\ \right\}\right] ,X_4\right) +\omega \left(\left[\left\{x^{1} ,\ \right\} ,X_4\right] ,\left\{x^{3} ,\ \right\}\right) -\omega \left(\left[\left\{x^{3} ,\ \right\} ,X_4\right] ,\left\{x^{1} ,\ \right\}\right)\\
&=\left\{x^{1} ,X_4\left( x^{3}\right)\right\} -\left\{x^{3} ,X_4\left( x^{1}\right)\right\} +X_4\left(\left\{x^{3} ,\ x^{1} \ \right\}\right)-\omega \left(\left\{\left\{x^{1} ,x^{3}\right\} ,\ \right\} ,X_4\right) +\omega \left( -\{x^1, \ \} ,\left\{x^{3} ,\ \right\}\right) -0\\
&=\left\{x^{1} ,x^{3}\right\} -2\left\{x^{3} ,x^{1}\right\} -X_4\left( 2x^{1}\right) -X_4\left( 2x^{1}\right) +\left\{x^{1} ,x^{3}\right\}=4\left\{x^{1} ,x^{3}\right\} -8x^{1}=0,
\end{align*}
where we have used
\begin{align*}
\left[\left\{x^{3} ,\ \right\} ,X_4\right] &=\left[ 2x^{2}\frac{\partial }{\partial x^{2}} -2x^{1}\frac{\partial }{\partial x^{1}} ,2x^{1}\frac{\partial }{\partial x^{1}} +x^{3}\frac{\partial }{\partial x^{3}}\right] =0,\ \  X_4\left( x^{3}\right) =x^{3} ,\\
\left[\left\{x^{1} ,\ \right\} ,\left\{x^{3} ,\ \right\}\right] &=\left\{x^{1} ,\left\{x^{3} ,\ \right\}\right\} -\left\{x^{3} ,\left\{x^{1} ,\ \right\}\right\} =\left\{\left\{x^{1} ,x^{3}\right\} ,\ \right\}.
\end{align*}
Finally, we calculate
\begin{align*}
&\left(\operatorname{d}_{\operatorname{dR}} \omega \right)\left(\left\{x^{2} ,\ \right\} ,\left\{x^{3} ,\ \right\} ,X_4\right)\\
&=\left\{x^{2} ,\omega \left(\left\{x^{3} ,\ \right\} ,X_4\right)\right\} -\left\{x^{3} ,\omega \left(\left\{x^{2} ,\ \right\} ,X_4\right)\right\} +X_4\omega \left(\left\{x^{2} ,\ \right\} ,\left\{x^{3} ,\ \right\}\right)\\
&\ \ \ \ \ \ \ \ \ \ \ \ \ \ \ \ \ \ \ \ \ \ -\omega \left(\left[\left\{x^{2} ,\ \right\} ,\left\{x^{3} ,\ \right\}\right] ,X_4\right) +\omega \left(\left[\left\{x^{2} ,\ \right\} ,X_4\right] ,\left\{x^{3} ,\ \right\}\right) -\omega \left(\left[\left\{x^{3} ,\ \right\} ,X_4\right] ,\left\{x^{2} ,\ \right\}\right)\\
&=\left\{x^{2} ,X_4\left( x^{3}\right)\right\} -\left\{x^{3} ,X_4\left( x^{2}\right)\right\} +X_4\left(\left\{x^{2} ,\ x^{3} \ \right\}\right)-\omega \left(\left\{\left\{x^{2} ,x^{3}\right\} ,\ \right\} ,X_4\right) +\omega \left(\left\{x^{2} ,\ \right\},\left\{x^{3} ,\ \right\}\right)\\
&=\left\{x^{2} ,x^{3}\right\} -0+0+\omega \left(\left\{2x^{2} ,\ \right\} ,X_4\right) + \left\{x^{3} ,x^{2}\right\}=0,
\end{align*}
proving that $ \operatorname{d}_{\operatorname{dR}} \omega =0$.

To check non-degeneracy we need to calculate the kernel of the matrix
\begin{align}
\label{eq:omega}
 \begin{array}{ c||c| c| c| c }
\omega ( \ ,\ ) & \left\{x^{1} ,\ \right\} & \left\{x^{2} ,\ \right\} & \left\{x^{3} ,\ \right\} & X_4\\
\hline\hline
\left\{x^{1} ,\ \right\} & 0 & -4x^{3} & -2x^{1} & 2x^{1}\\\hline
\left\{x^{2} ,\ \right\} & 4x^{3} & 0 & 2x^{2} & 0\\\hline
\left\{x^{3} ,\ \right\} & 2x^{1} & -2x^{2} & 0 & x^{3}\\\hline
X_4 & -2x^{1} & 0 & -x^{3} & 0
\end{array}
\end{align}
over $ A=P/( f)$. By inspection the columns of the matrix
\begin{align*}
    \left[
\begin{matrix}
x^{2} & 0 & 0 & -x^{3}\\
0 & x^{1} & x^{3} & 0\\
-2x^{3} & 0 & 0 & 2x^{1}\\
-2x^{3} & 2x^{3} & 2x^{2} & 2x^{1}
\end{matrix}\right],
\end{align*}
representing the differential of the resolution  \eqref{eq:res} in homological degree $1$, are also generators for $\ker(\omega)$.
Note that when removing the last row from the matrix \eqref{eq:omega} one the matrix presenting $\operatorname{Der}(A)$.
%By the exact sequence \eqref{eq:res} the transposed gradient vector of $f$,
%$\begin{bmatrix}
%x^{2}&
%x^{1}&
%-2x^{3} &
%0
%\end{bmatrix}^\top,$
%generates the kernel when the matrix is interpreted as a $2$-form on the free module $A^4$.
This means that the $2$-form $\omega$ is well-defined on $\operatorname{Der}(A)$ and non-degenerate.  We have no explanation for the fact that its determinant is $ ( 4f)^{2}$.

We have proven that $\omega$ is a symplectic form on the double cone.

\section{Explanation in terms of orbifold symplectic forms}\label{sec:expl}

The aim is here to interprete our symplectic form on the double cone as an instance of an orbifold symplectic form. We use results of Gerald Schwarz's article about lifting homotopies from orbit spaces \cite{LiftingHomo}.

Let $ \Gamma $ be a finite group acting linearly on the $ \boldsymbol{k}$-vector space $ V$. Consider the cotangent lifted action of $ \Gamma $ on $ T^{*} V=V\times V^{*}$ and denote by $ R$ the coordinate ring $ \boldsymbol{k}\left[ T^{*} V\right]$. Using a fundamental system $ u_{1} ,\dotsc ,u_{k} \in R^{\Gamma }$ of invariants we obtain an isomorphism of $ \boldsymbol{k}$-algebras $ R^{\Gamma } \simeq P/I$ where $ P=\boldsymbol{k}[ x_{1} ,\dotsc ,x_{n}]$ is a graded polynomial ring and the ideal $ I$ is given by the relations among the
$ u_{1} ,\dotsc ,u_{k}$. The group $ \Gamma $ acts on a differential form $ \omega \in \bigwedge{}_{R} \Omega _{R|\boldsymbol{k}}$ on the left by pullback $ \omega \mapsto \left( \gamma ^{-1}\right)^{*} \omega $, $ \gamma \in \Gamma $.
We write $ \left(\bigwedge{}_{R} \Omega _{R|\boldsymbol{k}}\right)^{\Gamma }$ for the $ R^{\Gamma }$-module of $\Gamma$-invariant differential forms. The group $ \Gamma $ acts also on derivations $ X\in \operatorname{Der}( R)$ by $ X\mapsto \left( \gamma ^{-1}\right)^{*} X\ \gamma ^{*}$. The $ R^{\Gamma }$-module of $\Gamma$-invariant derivations will be denoted by $ \operatorname{Der}( R)^{\Gamma }$. The space of $\Gamma$-invariant differential forms $ \left(\bigwedge{}_{R} \Omega _{R|\boldsymbol{k}}\right)^{\Gamma }$ can be interpreted as $ R^{\Gamma }$-valued alternating multilinear maps, i.e., we have a map
$$ \left(\bigwedge{}_{R} \Omega _{R|\boldsymbol{k}}\right)^{\Gamma }\rightarrow \operatorname{Alt}_{R^{\Gamma }}\left(\operatorname{Der}( R)^{\Gamma } ,R^{\Gamma }\right).$$
Since $ \left(\bigwedge{}_{R} \Omega _{R|\boldsymbol{k}}\right)^{\Gamma }$is the $ \boldsymbol{k}$-vector subspace of the free module $ \bigwedge{}_{R} \Omega _{R|\boldsymbol{k}} \simeq \operatorname{Alt}_{R}\left(\operatorname{Der}( R) ,R\right)$ this map is injective. The canonical symplectic form $ \omega _{0}$ on $ T^{*} V$ is actually in $ \left(\bigwedge{}_{R}^{2} \Omega _{R|\boldsymbol{k}}\right)^{\Gamma }$.

The invariant differential forms $ \left( C_{\operatorname{dR}}^{\Gamma }( R) :=\left(\bigwedge{}_{R} \Omega _{R|\boldsymbol{k}}\right)^{\Gamma } ,\operatorname{d_{dR}}\right)$ form a subcomplex, the \textit{invariant de Rham complex}, of the usual de Rham complex $ \left( C_{\operatorname{dR}}( R) :=\bigwedge{}_{R} \Omega _{R|\boldsymbol{k}} ,\operatorname{d_{dR}}\right)$ since $ \ \gamma ^{*}\operatorname{d_{dR}} =\operatorname{d_{dR}} \gamma ^{*}$ for each $ \gamma \in \Gamma $. Using the $ R^{\Gamma }$-linear embedding $ \left(\bigwedge{}_{R} \Omega _{R|\boldsymbol{k}}\right)^{\Gamma }\rightarrow \operatorname{Alt}_{R^{\Gamma }}\left(\operatorname{Der}( R)^{\Gamma } ,R^{\Gamma }\right)$ the differential on the right hand side can be expressed by the Koszul formula
\begin{align*}
\left(\operatorname{d}_{\operatorname{dR}} \omega \right)( Y_{0} ,Y_{1} ,\dotsc ,Y_{n}) &=\sum _{i=0}^{n}( -1)^{i} Y_{i}\left( \omega \left( Y_{0} ,\dotsc ,\widehat{Y_{i}} ,\dotsc ,Y_{n}\right)\right)\\
&\ \ \ \ \ \ \ \ \ \ \ \ \ \ \ \ \ \ \ \ \ \ \ \ \  +\sum _{0\leq i< j\leq n}( -1)^{i+j} \omega \left([ Y_{i} ,Y_{j}] ,Y_{0} ,\dotsc ,\widehat{Y_{i}} ,\dotsc ,\widehat{Y_{j}} ,\dotsc ,Y_{n}\right) ,
\end{align*}
where $ Y_{0} ,\dotsc ,Y_{n} \in \operatorname{Der}( R)^{\Gamma }$. We will show that when $ \boldsymbol{k} =\mathbb{C}$ the canonical $ R^{\Gamma }$-linear map $ \pi :\operatorname{Der}( R)^{\Gamma }\rightarrow \operatorname{Der}\left( R^{\Gamma }\right)$ is an isomorphism \ so that $ \omega _{0}$ can be interpreted as a symplectic form in the naive de Rham complex since, using $ R^{\Gamma } \simeq P/I$, $ \left( C_{\operatorname{dR}}^{\Gamma }( R) :=\left(\bigwedge \Omega _{R|\boldsymbol{k}}\right)^{\Gamma } ,\operatorname{d_{dR}}\right)$ can be seen as a subcomplex of the naive de Rham complex $ \left(\operatorname{Alt}_{P/I}\left(\operatorname{Der}( P/I) ,P/I\right) ,\operatorname{d}_{\operatorname{dR}}\right)$.

\begin{proposition}\label{prop:lift}
If $ \boldsymbol{k} =\mathbb{C}$ then the canonical $ R^{\Gamma }$-linear map $ \pi :\operatorname{Der}( R)^{\Gamma }\rightarrow \operatorname{Der}\left( R^{\Gamma }\right)$ is onto.
\end{proposition}

\begin{proof}
This follows from \cite[Corollary 7.7]{LiftingHomo}.
\end{proof}

\begin{theorem} \label{thm:inj}
If $ \boldsymbol{k} =\mathbb{C}$ then the canonical $ R^{\Gamma }$-linear map $ \pi :\operatorname{Der}( R)^{\Gamma }\rightarrow \operatorname{Der}\left( R^{\Gamma }\right)$ is injective.
\end{theorem}

Following \cite{LiftingHomo} we define $ \operatorname{Der}_{\Gamma }( R)^{\Gamma } :=\left\{X\in \operatorname{Der}( R)^{\Gamma } \ |\ X_{|R^{\Gamma }} =0\right\}$ as the subspace of $ \Gamma $-invariant derivations that vanish on $ \Gamma $-invariant functions one derives a short exact sequence of Lie-Rinehart algebras
\begin{align}
0\leftarrow \operatorname{Der}\left( R^{\Gamma }\right)\xleftarrow{\pi }\operatorname{Der}( R)^{\Gamma }\leftarrow \operatorname{Der}_{\Gamma }( R)^{\Gamma }\leftarrow 0.
\end{align}
Here the Lie algebras are viewed as $ R^{\Gamma }$-modules. By Proposition \ref{prop:lift} this sequence splits in the category of $ R^{\Gamma }$-modules, i.e., there is a $ R^{\Gamma }$-linear map
\begin{align}\label{eq:lift}
\lambda :\operatorname{Der}\left( R^{\Gamma }\right)\rightarrow \operatorname{Der}( R)^{\Gamma }
\end{align}
such that $ \pi \lambda =\operatorname{id}_{\operatorname{Der}\left( R^{\Gamma }\right)}$. In fact, such a lift can be constructed by choosing generators of
$ \operatorname{Der}\left( R^{\Gamma }\right)$ and lifting each of them to $ \operatorname{Der}( R)^{\Gamma }$ according to Proposition \ref{prop:lift}.
The proof of Theorem \ref{thm:inj} is a consequences of the following observation.

\begin{lemma} \label{lem:Chris}
Let $\boldsymbol{k}=\mathbb R$ or $\mathbb C$.
Assume $ \Gamma \rightarrow \operatorname{End}( U)$ to be a representation of the finite group $ \Gamma $ on the finite dimensional $ \boldsymbol{k}$-vector space $ U$ and let $ S=\boldsymbol{k}[ U]$. Then any $ X\in \operatorname{Der}( S)$ such that $ X_{|S^{\Gamma }} =0$ has to vanish.
\end{lemma}

\begin{proof}
Let $ X\in \operatorname{Der}( S)$ be such that $ X_{|S^{\Gamma }} =0$. We argue by contradiction. Choose linear coordinates $ z^{1} ,\dotsc z^{n}$ on $ U$ and write
$$ X=\sum _{j} X^{j}\frac{\partial }{\partial z^{j}}$$ with $ X^{j} \in S,\ j=1,\dotsc ,n$. Assume that there is a $ u_{0} \in U$ such that not all $ X^{j}( u_{0}) ,\ j=1,\dotsc ,n$ vanish. There is an analytic integral curve $ \boldsymbol{z} :D_{\delta }\rightarrow U,\ t\mapsto \boldsymbol{z}( t) ,$ defined on an open discs $ D_{\delta } \subseteq \boldsymbol{k}$ of radius $ \delta \in ] 0,\infty ]$ around $ 0$ such that
$ \boldsymbol{z}( 0) =u_{0}$ and $ d\boldsymbol{z}^{j}( t) /dt =X^{j}(\boldsymbol{z}( t))$ for all $ j=1,\dotsc ,n$ and $ t\in D_{\delta }$ (see, e.g., \cite{IlyaYak}). By our assumption $ t\mapsto \boldsymbol{z}( t)$ is non-constant. Now an $ f\in S^{\Gamma }$ has to be constant on $ \boldsymbol{z}( D_{\delta })$, which is open by the Open Mapping Theorem (see, e.g., \cite{Stein}). The categorical quotient $ U/\!\!/\Gamma =\operatorname{Spec}\left( S^{\Gamma }\right)$ separates closed orbits and, since $ \Gamma $ is finite, all orbits are closed. Since $ \boldsymbol{z}( D_{\delta })$
is open there is a $ u\in \boldsymbol{z}( D_{\delta }) \backslash ( \Gamma u_{0})$. But this means that there is an $ f\in S^{\Gamma }$ with $ f( u) \neq f( u_{0})$, contradicting our assumption.
\end{proof}

\begin{corollary}\label{cor:nodeg}
The restriction $\omega _{0}{}_{|\operatorname{Der}( R)^{\Gamma }}$ of the canonical symplectic form $\omega _{0}$ to the $ R^{\Gamma }$-module $\operatorname{Der}( R)^\Gamma$ is non-degenerate.
\end{corollary}

\begin{proof}
We will prove $ \ker( \omega _{0}{}_{|\operatorname{Der}( P)^{\Gamma }}) \subseteq \operatorname{Der}_{\Gamma }( R)^{\Gamma }$, so that the claim follows from Theorem \ref{thm:inj}. Let $ X\in \operatorname{Der}( R)^{\Gamma } \ $. We want to show that $\omega _{0}( X,Y) =0$ for all $ Y\in \operatorname{Der}( R)^{\Gamma }$ implies $ X_{|R^{\Gamma }} =0$. But for all $ f\in R^{\Gamma }$ we have $ 0=\omega _{0}( X,\{f,\ \}) =-X( f)$.
\end{proof}

\begin{corollary}\label{cor:sympform}
Using the lift $ \lambda $ of Equation \eqref{eq:lift} and assuming $ X,Y\in \operatorname{Der}\left( R^{\Gamma }\right)$
$$ \omega ( X,Y) :=\omega _{0}( \lambda ( X) ,\lambda ( Y))$$
defines a non-degenerate closed form $ \omega \in \operatorname{Alt}_{R^{\Gamma }}^{2}\left(\operatorname{Der}\left( R^{\Gamma }\right) ,R^{\Gamma }\right)$.
\end{corollary}

In the case of the double cone, seen as a symplectic $ \mathbb{Z}_{2}$-quotient, we put $ V=\mathbb{C}$ so that $ T^{*} V\simeq \mathbb{C}^{2} .$ We use canonical coordinates $ ( q,p)$ for $ T^{*} V$ \ so that $ \mathbb{C}\left[ T^{*} V\right] \simeq \mathbb{C}[ q,p] =:R$. The \ $ \mathbb{Z}_{2}$-action in terms of the canonical coordinates is given by $ ( q,p) \mapsto ( -q,-p)$ and the invariants are $ u_{1} =q^{2} ,\ u_{2} =p^{2}$ and $ u_{3} =qp$ obeying the relation $ u_{3}^{2} -u_{1} u_{2} =0$. Then $ \operatorname{Der}( R)^{\Gamma }$ is generated as an $ R^{\Gamma }$-module by the vector fields that are obtained by substituting
simultaneously in each monomial of $ u_{1} ,u_{2} ,u_{3}$ a $ q$ by $ \partial /\partial p$ or by substituting
simultaneously in each monomial of $ u_{1} ,u_{2} ,u_{3}$ a $ p$ by $ \partial /\partial q$. This way we find the four invariant vector fields
\begin{align}
\begin{array}{c||c|c|c}
 & u_{1} & u_{2} & u_{3}\\
\hline\hline
q\frac{\partial }{\partial q} & 2u_{1} & 0 & u_{3}\\\hline
q\frac{\partial }{\partial p} & 0 & 2u_{3} & u_{1}\\\hline
p\frac{\partial }{\partial q} & 2u_{3} & 0 & u_{2}\\\hline
p\frac{\partial }{\partial p} & 0 & 2u_{2} & u_{3}
\end{array}
\end{align}
which can be hence interpreted in terms of $ P=\mathbb{C}\left[ x^{1} ,x^{2} ,x^{3}\right] /\left( x^{3} x^{3} -x^{1} x^{2}\right)$ as
\begin{align*}
q\frac{\partial }{\partial q} \mapsto 2x^{1}\frac{\partial }{\partial x^{1}} +x^{3}\frac{\partial }{\partial x^{3}} ,\ \ q\frac{\partial }{\partial p} \mapsto 2x^{3}\frac{\partial }{\partial x^{2}} +x^{1}\frac{\partial }{\partial x^{3}} ,\ \\
p\frac{\partial }{\partial q} \mapsto 2x^{3}\frac{\partial }{\partial x^{1}} +x^{2}\frac{\partial }{\partial x^{3}} \ ,\ p\frac{\partial }{\partial p} \mapsto 2x^{2}\frac{\partial }{\partial x^{2}} +x^{3}\frac{\partial }{\partial x^{3}} .
\end{align*}
The base change to the system of generators of $ \operatorname{Der}_{I}( P)$ of the previous section is
\begin{align*}
\begin{bmatrix}
2x^{1} & 0 & 2x^{3} & 0\\
0 & 2x^{3} & 0 & 2x^{2}\\
x^{3} & x^{1} & x^{2} & x^{3}
\end{bmatrix}\begin{bmatrix}
0 & 0 & -1 & 1\\
2 & 0 & 0 & 0\\
0 & -2 & 0 & 0\\
0 & 0 & 1 & 0
\end{bmatrix} =\begin{bmatrix}
0 & -4x^{3} & -2x^{1} & 2x^{1}\\
4x^{3} & 0 & 2x^{2} & 0\\
2x^{1} & -2x^{2} & 0 & x^{3}
\end{bmatrix}.
\end{align*}
But from this we deduce the relations
\begin{align*}
0&=\begin{bmatrix}
0 & -4x^{3} & -2x^{1} & 2x^{1}\\
4x^{3} & 0 & 2x^{2} & 0\\
2x^{1} & -2x^{2} & 0 & x^{3}
\end{bmatrix}
\begin{bmatrix}
x^{2} & 0 & 0 & -x^{3}\\
0 & x^{1} & x^{3} & 0\\
-2x^{3} & 0 & 0 & 2x^{1}\\
-2x^{3} & 2x^{3} & 2x^{2} & 2x^{1}
\end{bmatrix}\\
&=\begin{bmatrix}
2x^{1} & 0 & 2x^{3} & 0\\
0 & 2x^{3} & 0 & 2x^{2}\\
x^{3} & x^{1} & x^{2} & x^{3}
\end{bmatrix}\begin{bmatrix}
0 & 0 & -1 & 1\\
2 & 0 & 0 & 0\\
0 & -2 & 0 & 0\\
0 & 0 & 1 & 0
\end{bmatrix}
\begin{bmatrix}
x^{2} & 0 & 0 & -x^{3}\\
0 & x^{1} & x^{3} & 0\\
-2x^{3} & 0 & 0 & 2x^{1}\\
-2x^{3} & 2x^{3} & 2x^{2} & 2x^{1}
\end{bmatrix}\\
&=\begin{bmatrix}
2x^{1} & 0 & 2x^{3} & 0\\
0 & 2x^{3} & 0 & 2x^{2}\\
x^{3} & x^{1} & x^{2} & x^{3}
\end{bmatrix}\begin{bmatrix}
0 & 2x^{3} & 2x^{2} & 0\\
2x^{2} & 0 & 0 & -2x^{3}\\
0 & -2x^{1} & -2x^{3} & 0\\
-2x^{3} & 0 & 0 & 2x^{1}
\end{bmatrix}
\end{align*}
%But from this we deduce the single relation
%\begin{align*}
%0&=\begin{bmatrix}
%0 & -4x^{3} & -2x^{1} & 2x^{1}\\
%4x^{3} & 0 & 2x^{2} & 0\\
%2x^{1} & -2x^{2} & 0 & x^{3}
%\end{bmatrix}\begin{bmatrix}
%x^{2}\\
%x^{1}\\
%-2x^{3}\\
%0
%\end{bmatrix} =\begin{bmatrix}
%2x^{1} & 0 & 2x^{3} & 0\\
%0 & 2x^{3} & 0 & 2x^{2}\\
%x^{3} & x^{1} & x^{2} & x^{3}
%\end{bmatrix}\begin{bmatrix}
%0 & 0 & -1 & 1\\
%2 & 0 & 0 & 0\\
%0 & -2 & 0 & 0\\
%0 & 0 & 1 & 0
%\end{bmatrix}\begin{bmatrix}
%x^{2}\\
%x^{1}\\
%-2x^{3}\\
%0
%\end{bmatrix}\\
%&=\begin{bmatrix}
%2x^{1} & 0 & 2x^{3} & 0\\
%0 & 2x^{3} & 0 & 2x^{2}\\
%x^{3} & x^{1} & x^{2} & x^{3}
%\end{bmatrix}\begin{bmatrix}
%2x^{3}\\
%2x^{2}\\
%-2x^{1}\\
%-2x^{3}
%\end{bmatrix}.
%\end{align*}
Expressed in terms of $ q\frac{\partial }{\partial q} ,\ q\frac{\partial }{\partial p} ,\ p\frac{\partial }{\partial q} ,\ p\frac{\partial }{\partial p}$ this relation turns out to be trivial
\begin{align*}
0&=2p^{2} \ q\frac{\partial }{\partial p} -2qp\ \ p\frac{\partial }{\partial p},\ \ \ 0=2qp\ q\frac{\partial }{\partial q} -2q^{2} \ p\frac{\partial }{\partial q} ,\\
0&=2p^{2} \ q\frac{\partial }{\partial q} -2qp\ p\frac{\partial }{\partial q} ,\ \ \ 0=-2qp\ \ q\frac{\partial }{\partial p} +2q^{2} \ p\frac{\partial }{\partial p},
\end{align*}
%\begin{align*}
%2qp\ q\frac{\partial }{\partial q} +2p^{2} \ q\frac{\partial }{\partial p} -2q^{2} \ p\frac{\partial }{\partial q} -2qp\ p\frac{\partial }{\partial p} =0,
%\end{align*}
and we verify once again that $ \pi $ is an isomorphism. Since the symplectic form of the previous section is unique it has to coincide with the one of Corollary \ref{cor:sympform}.

\section{The simple cone as a Poisson differential space}\label{sec:Diff}

The double cone $ P=\mathbb{C}\left[ x^{1} ,x^{2} ,x^{3}\right] /\left( x^{3} x^{3} -x^{1} x^{2}\right)$ can be understood as the complexification $ \mathbb{C} \otimes _{\mathbb{R}}\mathbb{R}[\mathbb{C} /\mathbb{Z}_{2}]$ where $ \mathbb{Z}_{2}$ is seen as a subgroup of $ \operatorname{SU}_{2}$ and $ \mathbb{R}[\mathbb{C} /\mathbb{Z}_{2}]$ is the algebra of regular functions on the orbit space $ \mathbb{C} /\mathbb{Z}_{2}$. In fact, writing $ z=q+\sqrt{-1} p$ the fundamental invariants with respect to $ \mathbb{Z}_{2} \subset \operatorname{SU}_{2} \subset \operatorname{SL}_{2}(\mathbb{C})$, $ z\mapsto -z$, \ can be interpreted as the real invariants
$$ u_{1} =\frac{1}{4}( z+\overline{z})^{2} ,\ u_{2} =\frac{-1}{4}( z-\overline{z})^{2} ,\ u_{3} =\frac{-\sqrt{-1}}{4}\left( z^{2} -\overline{z}^{2}\right) .$$
The image $ \mathcal{X}$ of the orbit space $ \mathbb{C}/\mathbb{Z}_{2}$ under $ \boldsymbol{u} =( u_{1} ,u_{2} ,u_{3}) :\mathbb{C}\rightarrow \mathbb{R}^{3}$ is actually given as the semialgebraic set
%cs edited inequalities; just x^1 and x^2 \geq 0
%$$ x^{3} x^{3} -x^{1} x^{2} =0,\ x^{1} x^{1} ,x^{2} x^{2} \geq 0,$$
$$ x^{3} x^{3} -x^{1} x^{2} =0,\ x^{1} ,x^{2} \geq 0,$$
which is a simple cone. The orbit space $ \mathbb{C} /\mathbb{Z}_{2}$ and the subeuclidean space $ \mathcal{X}$ are actually Poisson diffeomorphic differential spaces, see \cite{FHS}. We use here the notion of a differential space in the sense of Sikorsky, see also, e.g., \cite{SniatyckiBook}. The diffeomorphism is provided by the so-called \textit{Hilbert embedding}
 $$ \underline{\boldsymbol{u}} :\left(\mathbb{C} /\mathbb{Z}_{2} ,\mathcal{C}^{\infty }(\mathbb{C})^{\mathbb{Z}_{2}}\right)\rightarrow \left(\mathcal{X} ,\mathcal{C}^{\infty }(\mathcal{X})\right) ,$$
where
$ \underline{\boldsymbol{u}}:\mathbb{C}/\mathbb{Z}_{2}\rightarrow \mathbb{R}^{3}$ denotes the map induced by $ \boldsymbol{u}$.

%cs added definition of \mathcal{Y} at the beginning of the paragraph.
Let $\mathcal{Y}\subseteq \R^3$ be the real algebraic variety given by $x^{3} x^{3} -x^{1} x^{2}$ so that
$ \mathbb{R}[\mathcal{Y}] =\mathbb{R}\left[ x^{1} ,x^{2} ,x^{3}\right] /\left( x^{3} x^{3} -x^{1} x^{2}\right)$.
The generators $ \left\{x^{1} ,\ \right\},\ \left\{x^{2} ,\ \right\},\ \left\{x^{3} ,\ \right\},\ X_{4}$ of the $ P$-module of derivations, being real,
$$ \operatorname{Der}_{\left( x^{3} x^{3} -x^{1} x^{2}\right)}\left( P/\left( x^{3} x^{3} -x^{1} x^{2}\right)\right)$$
actually coincide with the generators of the $ \mathbb{R}[\mathcal{Y}]$-module of derivations $\operatorname{Der}(\mathbb{R}[\mathcal{Y}])$.
%, where we wrote $ \mathbb{R}[\mathcal{Y}] =\mathbb{R}\left[ x^{1} ,x^{2} ,x^{3}\right] /\left( x^{3} x^{3} -x^{1} x^{2}\right)$.
We would like to investigate $ \operatorname{Der}\left(\mathcal{C}^{\infty }(\mathcal{X})\right)$. We start with the following observation.

\begin{lemma}
\label{lem:diffder}
Let $ \mathcal{I} $ be a closed ideal in $ \mathcal{C}^{\infty }\left(\mathbb{R}^{n}\right)$. Then
 $$ \operatorname{Der}\left(\mathcal{C}^{\infty }\left(\mathbb{R}^{n}\right) /\mathcal{I}\right) \simeq \operatorname{Der}_{\mathcal{I}}\left(\mathcal{C}^{\infty }\left(\mathbb{R}^{n}\right)\right) /\mathcal{I}\operatorname{Der}\left(\mathcal{C}^{\infty }\left(\mathbb{R}^{n}\right)\right).$$
Here
%cs removed the Der subscript in the set brackets, i.e., after X\in
%$\operatorname{Der}_{\mathcal{I}}\left(\mathcal{C}^{\infty }\left(\mathbb{R}^{n}\right)\right) =\left\{X\in \operatorname{Der}_{\mathcal{I}}\left(\mathcal{C}^{\infty }\left(\mathbb{R}^{n}\right)\right) \ |\ X(\mathcal{I}) \subseteq \mathcal{I}\right\}$
$\operatorname{Der}_{\mathcal{I}}\left(\mathcal{C}^{\infty }\left(\mathbb{R}^{n}\right)\right) =\left\{X\in \operatorname{Der}\left(\mathcal{C}^{\infty }\left(\mathbb{R}^{n}\right)\right) \ |\ X(\mathcal{I}) \subseteq \mathcal{I}\right\}$
denotes the $ \mathcal{C}^{\infty }\left(\mathbb{R}^{n}\right)$-submodule of tangential derivations.
\end{lemma}

\begin{proof}
We adapt the argument of \cite[Lemma 2.1.2] {Simis} to the situation of differential spaces (see \cite{Navarro}). We have a map $\varphi :\operatorname{Der}_{\mathcal{I}}\left(\mathcal{C}^{\infty }\left(\mathbb{R}^{n}\right)\right)\rightarrow \operatorname{Der}\left(\mathcal{C}^{\infty }\left(\mathbb{R}^{n}\right) /\mathcal{I}\right)$ given by $ \varphi ( X) :=( f+\mathcal{I} \mapsto X( f) +\mathcal{I})$
for $ f\in \mathcal{C}^{\infty }\left(\mathbb{R}^{n}\right)$. Recall that $ \operatorname{Der}\left(\mathcal{C}^{\infty }\left(\mathbb{R}^{n}\right)\right)$ is the free $ \mathcal{C}^{\infty }\left(\mathbb{R}^{n}\right)$-module generated by the $ \partial /\partial x^{i} ,\ i=1,\dotsc ,n$ (see, e.g., \cite{Michor} or \cite{Navarro}). Now assume $ X=\sum _{i=1}^{n} X^{i}\frac{\partial }{\partial x^{i}} \in \ker( \varphi )$ with $ X^{i} \in \mathcal{C}^{\infty }\left(\mathbb{R}^{n}\right) ,\ i=1,\dotsc ,n$. This means that $X( f) \in \mathcal{I}$ for all $ f\in \mathcal{C}^{\infty }\left(\mathbb{R}^{n}\right)$. Choosing $ f=x^{j}$ we find $ X^{j} \in \mathcal{I}$. This shows that $ \ker( \varphi ) =\mathcal{I}\operatorname{Der}\left(\mathcal{C}^{\infty }\left(\mathbb{R}^{n}\right)\right)$.

It remains to prove that $ \varphi $ is onto. We now rely on the fact that $ \mathcal{C}^{\infty }\left(\mathbb{R}^{n}\right)$ is a nuclear Fréchet algebra. Let $ \Delta $ be the kernel of the multiplication morphism $ \mathcal{C}^{\infty }\left(\mathbb{R}^{n}\right)\hat{\otimes }_{\mathbb{R}}\mathcal{C}^{\infty }\left(\mathbb{R}^{n}\right)\rightarrow \mathcal{C}^{\infty }\left(\mathbb{R}^{n}\right)$
%cs moved explanation of \hat{\otimes} here, the first time it is used
where $ \hat{\otimes }$ denotes the completed tensor product.
It is an ideal in $ \mathcal{C}^{\infty }\left(\mathbb{R}^{n}\right)\hat{\otimes }_{\mathbb{R}}\mathcal{C}^{\infty }\left(\mathbb{R}^{n}\right) \simeq \mathcal{C}^{\infty }\left(\mathbb{R}^{2n}\right)$.
The \textit{module of differentials} $ \mathcal{D}_{\left(\mathcal{C}^{\infty }\left(\mathbb{R}^{n}\right) /\mathcal{I}\right) \ |\mathbb{R}}$ is then defined to be $ \Delta /\overline{\Delta ^{2}}$ (see \cite[Subsection 10.1]{Navarro}).
Here
%$ \hat{\otimes }$ denotes the completed tensor product and
$ \overline{\Delta ^{2}}$ is the closure of $ \Delta ^{2}$ in the Fréchet topology.
The \emph{second fundamental sequence} for the module
 $\mathcal{D}_{\left(\mathcal{C}^{\infty }\left(\mathbb{R}^{n}\right) /\mathcal{I}\right) \ |\mathbb{R}}$ of differentials (see \cite[Subsection 10.2]{Navarro}) is given by
$$ 0\leftarrow \mathcal{D}_{\left(\mathcal{C}^{\infty }\left(\mathbb{R}^{n}\right) /\mathcal{I}\right) \ |\mathbb{R}}\leftarrow \left(\mathcal{C}^{\infty }\left(\mathbb{R}^{n}\right) /\mathcal{I}\right)\hat{\otimes }_{\mathcal{C}^{\infty }\left(\mathbb{R}^{n}\right)}\mathcal{D}_{\mathcal{C}^{\infty }\left(\mathbb{R}^{n}\right) \ |\mathbb{R}}\xleftarrow{d}\mathcal{I} /\overline{\mathcal{I}^{2}}$$
%cs removed second explanation of \hat{\otimes}
%Here $ \hat{\otimes }$ denotes the completed tensor product and
where $ \overline{\mathcal{I}^{2}}$ is the closure of $ \mathcal{I}^{2}$ in the Fréchet topology. We apply the functor $ \operatorname{Hom}_{\mathcal{C}^{\infty }\left(\mathbb{R}^{n}\right) /\mathcal{I}}\left( \ \ ,\mathcal{C}^{\infty }\left(\mathbb{R}^{n}\right) /\mathcal{I}\right)$ to this sequence and recall (see \cite[Theorem 10.4]{Navarro}) that $ \operatorname{Hom}_{\mathcal{C}^{\infty }\left(\mathbb{R}^{n}\right) /\mathcal{I}}\left( \ \mathcal{D}_{\left(\mathcal{C}^{\infty }\left(\mathbb{R}^{n}\right) /\mathcal{I}\right) \ |\mathbb{R}} ,\mathcal{C}^{\infty }\left(\mathbb{R}^{n}\right) /\mathcal{I}\right) \simeq \operatorname{Der}\left(\mathcal{C}^{\infty }\left(\mathbb{R}^{n}\right) /\mathcal{I}\right)$.
As a result we obtain that
$\operatorname{Der}\left(\mathcal{C}^{\infty }\left(\mathbb{R}^{n}\right) /\mathcal{I}\right)$ is the kernel of the map
%cs added \sloppy to fix line breaks
\sloppy $ \operatorname{Hom}_{\mathcal{C}^{\infty }\left(\mathbb{R}^{n}\right) /\mathcal{I}}\left(\mathcal{I} /\overline{\mathcal{I}^{2}} ,\mathcal{C}^{\infty }\left(\mathbb{R}^{n}\right) /\mathcal{I}\right)\xleftarrow{d^{\lor }} \bigoplus _{i=1}^{n}\left(\mathcal{C}^{\infty }\left(\mathbb{R}^{n}\right) /\mathcal{I}\right)\frac{\partial }{\partial x^{i}} ,$
where we used the isomorphisms
\begin{align*} \left(\mathcal{C}^{\infty }\left(\mathbb{R}^{n}\right) /\mathcal{I}\right)\hat{\otimes }_{\mathcal{C}^{\infty }\left(\mathbb{R}^{n}\right)}\mathcal{D}_{\mathcal{C}^{\infty }\left(\mathbb{R}^{n}\right) \ |\mathbb{R}} &\simeq \bigoplus _{i=1}^{n}\left(\mathcal{C}^{\infty }\left(\mathbb{R}^{n}\right) /\mathcal{I}\right) dx^{i},\\
\operatorname{Hom}_{\mathcal{C}^{\infty }\left(\mathbb{R}^{n}\right) /\mathcal{I}}\left( \bigoplus _{i=1}^{n}\left(\mathcal{C}^{\infty }\left(\mathbb{R}^{n}\right) /\mathcal{I}\right) dx^{i} ,\mathcal{C}^{\infty }\left(\mathbb{R}^{n}\right) /\mathcal{I}\right)& \simeq \bigoplus _{i=1}^{n}\left(\mathcal{C}^{\infty }\left(\mathbb{R}^{n}\right) /\mathcal{I}\right)\frac{\partial }{\partial x^{i}}.
\end{align*}
In other words, we can understand an $ X\in \operatorname{Der}\left(\mathcal{C}^{\infty }\left(\mathbb{R}^{n}\right) /\mathcal{I}\right)$ as an element $ \sum _{i=1}^{n}\left( X^{i} +\mathcal{I}\right)\frac{\partial }{\partial x^{i}}$ with $ X^{i} \in \mathcal{C}^{\infty }\left(\mathbb{R}^{n}\right)$ from $ \bigoplus _{i=1}^{n}\left(\mathcal{C}^{\infty }\left(\mathbb{R}^{n}\right) /\mathcal{I}\right)\frac{\partial }{\partial x^{i}}$ which has to be also in $ \operatorname{ker}\left( d^{\lor }\right)$. This means that $ \sum _{i=1}^{n} X^{i}\frac{\partial f} {\partial x^{i}} \in \mathcal{I}$ \ for each $f\in \mathcal I$. Hence the representative $ X=\sum _{i=1}^{n} X^{i}\frac{\partial }{\partial x^{i}}$ is in $ \operatorname{Der}_{\mathcal{I}}\left(\mathcal{C}^{\infty }\left(\mathbb{R}^{n}\right)\right)$ and we have proven that $ \varphi $ is onto.
\end{proof}
A theorem of B. Malgrange, \cite[Chapter VI, Th. 1.1']{Malgrange}, says that any ideal $ ( f_{1} ,\dotsc ,f_{k})$ in \ $ \mathcal{C}^{\infty }\left(\mathbb{R}^{n}\right)$ generated by real analytic functions $ f_{1} ,\dotsc ,f_{k}$ is closed. Hence we can apply the lemma to affine varieties.
Moreover, if $\mathcal{X}$ is a Nash subanalytic set (e.g. a semialgebraic set), then its vanishing ideal is closed. It is even forms a complemented subspace in the Fréchet algebra $\mathcal{C}^{\infty }\left(\mathbb{R}^{n}\right)$ \cite[Theorem 0.2.1]{BierSchwDuke}.

\begin{theorem}\label{thm:smoothder}
Let $\mathcal{Y} \subseteq \mathbb{R}^{n}$ the affine $\mathbb R$-variety given by the equations $f_{r} =0,\ r=1,\dotsc ,k$. Write $\mathbb{R} [\mathcal{Y} ]=\mathbb{R}\left[ x^{1} ,\dotsc ,x^{n}\right] /(f_{1} ,\dotsc ,f_{k} )$ and $\mathcal{C}^{\infty } (\mathcal{Y} )=\mathcal{C}^{\infty }\left(\mathbb{R}^{n}\right) /(f_{1} ,\dotsc ,f_{k} )$. If $X_{1} ,\dotsc ,\ X_{l}$ generate $\operatorname{Der} (\mathbb{R} [\mathcal{Y} ])$ as an $\mathbb{R} [\mathcal{Y} ]$-module then they also generate $\operatorname{Der}\left(\mathcal{C}^{\infty } (\mathcal{Y} )\right)$ as a $\mathcal{C}^{\infty } (\mathcal{Y} )$-module.
\end{theorem}
\begin{proof}
In order to determine the generators of the tangential derivations $\operatorname{Der}_{(f_{1} ,\dotsc ,f_{k} )}\left(\mathcal{C}^{\infty }\left(\mathbb{R}^{n}\right)\right)$ we use the same argument as in Proposition \ref{prop:DerGen}. We can do this because by Lemma \ref{lem:diffder} we have
\begin{equation*}
\operatorname{Der}\left(\mathcal{C}^{\infty } (\mathcal{X} )\right) =\operatorname{Der}_{(f_{1} ,\dotsc ,f_{k} )}\left(\mathcal{C}^{\infty }\left(\mathbb{R}^{n}\right)\right) /(f_{1} ,\dotsc ,f_{k})\operatorname{Der}\left(\mathcal{C}^{\infty }\left(\mathbb{R}^{n}\right)\right) .
\end{equation*}
Accordingly, we need to determine the generators of $\ker(\mathcal{M})$ where
\begin{equation*}
\mathcal{M} :\mathcal{C}^{\infty }\left(\mathbb{R}^{n}\right)^{n+k}\rightarrow \mathcal{C}^{\infty }\left(\mathbb{R}^{n}\right)^{k} ,
\end{equation*}
is the matrix \eqref{eq:M}, seen as a morphism of free $\mathcal{C}^{\infty }\left(\mathbb{R}^{n}\right)$-modules. Write
$$M:\mathbb{R}[ x^{1} ,\dotsc ,x^{n}]^{n+k}\rightarrow \mathbb{R}[ x^{1} ,\dotsc ,x^{n}]^{k}$$
for the corresponding morphism of free $\mathbb{R}[x^{1} ,\dotsc ,x^{n}]$-modules. Let
$(F:=\bigoplus _{i\geq 0} F_{i} ,\delta )$
be a free resolution of the module $\operatorname{coker}( M)$. Note that the matrices representing the differentials  $\delta _{i} :F_{i}\rightarrow F_{i-1}$ of $F$ have polynomial entries. We want to show that $\mathcal{C}^{\infty }\left(\mathbb{R}^{n}\right) \otimes _{\mathbb{R}\left[ x^{1} ,\dotsc ,x^{n}\right]}\ker( M)$ is isomorphic to the $\mathcal{C}^{\infty }\left(\mathbb{R}^{n}\right)$-module. But, by a standard argument (see \cite{HS13}), the complex of free $\mathcal{C}^{\infty }\left(\mathbb{R}^{n}\right)$-modules $\mathcal{F:=C}^{\infty }\left(\mathbb{R}^{n}\right) \otimes _{\mathbb{R}\left[ x^{1} ,\dotsc ,x^{n}\right]} F$ is exact in homological degree $j\geq 1$.
\end{proof}

For semialgebraic sets there is in general no such a simple statement like Theorem \ref{thm:smoothder}. However, Edward Bierstone \cite{LiftingIso} and Gerald Schwarz \cite{LiftingHomo} investigated the case when the locally semialgebraic set $\mathcal{X}$ is isomorphic via a Hilbert map to the orbit space $N/K$ of a $K$-manifold $N$ by a compact Lie group $K$. (If $N$ is a $K$-module then $N/K$ is isomorphic to a semialgebraic set $\mathcal{X}$.) It turns out that only those derivations $X\in \operatorname{Der}\left(\mathcal{C}^{\infty }( N/K)\right)$ that are tangent to all orbit type strata can be lifted to $\operatorname{Der}\left(\mathcal{C}^{\infty }( N)\right)^{K}$
%cs added "Lemma 3.7" to citation
\cite[Lemma 3.7]{BierstoneIMPA}. They are called \textit{smooth vector fields on}
$N/K$ and the set of smooth vector fields will be denoted by $\mathfrak{X}^{\infty }( N/K)$. We use one of their results to investigate $\operatorname{Der}(\mathcal{X})$ for the simple cone $\mathcal{X}$.

\begin{proposition}{\cite[Proposition 3.9]{BierstoneIMPA}}
\label{prop:codim1}
Let $K$ be a compact Lie group, $N$ be a smooth $K$-manifold and $X\in \operatorname{Der}\left(\mathcal{C}^{\infty }( N/K)\right)$ . Then $X\in \mathfrak{X}^{\infty }( N/K)$ if and only
if $X$ is tangent to the codimension one strata in $N/K$. In particular, $\mathfrak{X}^{\infty }( N/K) =\operatorname{Der}\left(\mathcal{C}^{\infty }( N/K)\right)$
if and only if there are no codimension one strata in $N/K$.
\end{proposition}

In the case of the simple cone $\mathcal{X}$ there are no codimension one strata. By Theorem \ref{thm:smoothder} generators for the derivations of $\mathcal{C}^{\infty }\left(\mathbb{R}^{3}\right) /\left( x^{3} x^{3} -x^{1} x^{2}\right)$
are given by the representatives \eqref{eq:genDHam} and \eqref{eq:genEuler}. Derivations tangent to the union of strata
%cs edited definition of \mathcal{X}; x^1 \geq 0, not x^1 x^1, and similarly for x^2
%$$\mathcal{X} =\left\{\left( x^{1} ,x^{2} ,x^{3}\right) \in \mathbb{R}^{3} \ |\ x^{3} x^{3} -x^{1} x^{2} ,\ x^{1} x^{1} \geq 0,\ x^{2} x^{2} \geq 0\right\}$$
$$\mathcal{X} =\left\{\left( x^{1} ,x^{2} ,x^{3}\right) \in \mathbb{R}^{3} \ |\ x^{3} x^{3} -x^{1} x^{2} ,\ x^{1} \geq 0,\ x^{2} \geq 0\right\}$$
are in turn generated by (the infinitely many) vector fields of the form $Y_{i} =\sum _{j=1}^{3} Y_{i}^{j}\left( x^{1} ,x^{2} ,x^{3}\right)\frac{\partial }{\partial x^{j}}$ such that %cs added that
\begin{align*}
Y_{1}{}_{|\mathcal{H}} &=4x^{3}\frac{\partial }{\partial x^{2}} +2x^{1}\frac{\partial }{\partial x^{3}} ,\ \ Y_{2}{}_{|\mathcal{H}} =-4x^{3}\frac{\partial }{\partial x^{1}} -2x^{2}\frac{\partial }{\partial x^{3}} ,\\
Y_{3}{}_{|\mathcal{H}} &=2x^{2}\frac{\partial }{\partial x^{2}} -2x^{1}\frac{\partial }{\partial x^{1}} ,\ \ Y_{4}{}_{|\mathcal{H}} =2x^{1}\frac{\partial }{\partial x^{1}} +2x^{3}\frac{\partial }{\partial x^{3}} ,
\end{align*}
where
%cs edited definition of \mathcal{H}, same as \mathcal{X} above
%$\mathcal{H} =\left\{\left( x^{1} ,x^{2} ,x^{3}\right) \in \mathbb{R}^{3} \ |\ x^{1} x^{1} \geq 0,\ x^{2} x^{2} \geq 0\right\}$.
$\mathcal{H} =\left\{\left( x^{1} ,x^{2} ,x^{3}\right) \in \mathbb{R}^{3} \ |\ x^{1} \geq 0,\ x^{2} \geq 0\right\}$.
As the $Y_{j}$ vanish at $( 0,0,0)$ for $j=1,2,3,4$ they are tangent to the stratum $\{( 0,0,0)\} \subseteq \mathcal{X}$.
Modulo $\mathcal{I}\operatorname{Der}(\mathcal{X})$ the $Y_{j} ,\ j=1,2,3,4$ coincide with the four generators of \eqref{eq:genDHam} and \eqref{eq:genEuler}.

For the simple cone $\mathcal{X}$ we can now use the same formulas as in Section \ref{sec:dc} to construct a closed $2$-form
$$\omega \in \operatorname{Alt}_{\mathcal{C}^{\infty }( \mathcal X)}^{2}\left(\operatorname{Der}\left(\mathcal{C}^{\infty }(\mathcal{X})\right) ,\mathcal{C}^{\infty }(\mathcal{X})\right)$$ in the naive de Rham complex of the Lie-Rinehart
algebra $ \left(\operatorname{Der}\left(\mathcal{C}^{\infty }(\mathcal{X})\right) ,\mathcal{C}^{\infty }(\mathcal{X})\right)$. In order to see that it is non-degenerate we apply the functor
$ \mathcal{C}^{\infty }\left(\mathbb{R}^{n}\right) \otimes _{\mathbb{R}\left[ x^{1} ,x^2 ,x^{3}\right]}( \ )$ to the exact sequence \eqref{eq:res}. Using the same type of argument as in the proof of Theorem \ref{thm:smoothder} we see that the result is exact.

\section{Linear symplectic orbifolds}\label{sec:orbifold}

The considerations of Section \ref{sec:expl} can actually be adapted to find symplectic forms in the sense of our definition on linear symplectic orbifolds, generalizing the case of the simple cone (see Section \ref{sec:Diff}) and avoiding constructive invariant theory. By a linear \emph{symplectic orbifold} (see \cite{FHS, HSSorb, HSStorus}) we mean the differential space $\left( V/\Gamma ,\mathcal{C}^{\infty }( V)^{\Gamma }\right)$ where $V/\Gamma $ is the space of $\Gamma $-orbits in $V$ and $\Gamma \rightarrow \operatorname{U}( V)$ is a unitary representation of a finite group $\Gamma $ on a finite dimensional hermitean vector space $V$. By a theorem of H. Weyl \cite{Weyl} there is a fundamental system of invariants $u_{1} ,\dotsc ,u_{n} \in \mathbb{R}[ V]^{\Gamma }$ such that any other invariant $f\in \mathbb{R}[ V]^{\Gamma }$ can be written as the composite $f=F( u_{1} ,\dotsc ,u_{n})$ for some $F\in \mathbb{R}[ x^{1} ,\dotsc ,x^{n}]$. The map $\boldsymbol{u} =( u_{1} ,\dotsc ,u_{n}) :V\rightarrow \mathbb{R}^{n}$ defines an embedding of $\underline{\boldsymbol{u}} :V/\Gamma \rightarrow \mathbb{R}^{n}$, referred to as a \textit{Hilbert embedding}, whose image $\mathcal{X} :=\boldsymbol{u}( V)$ is a semialgebraic set. The Zarisky closure of $\mathcal{X}$ is described by the ideal $(f_{1} ,\dotsc ,f_{k})$ of relations among the $u_{1},\dotsc,u_{n}$ and the inequalities cutting out $\mathcal{X}$ from its Zarisky closure can be determined from $u_{1} ,\dotsc ,u_{n}$ according to a recipe of Claudio Procesi and Gerald Schwarz \cite{PS}. The theorem of Gerald Schwarz on differentiable invariants \cite{DI} says that any smooth $\Gamma $-invariant function $f\in \mathcal{C}^{\infty }( V)^{\Gamma }$ can be written as the composite $f=F( u_{1} ,\dotsc ,u_{n})$ for some $F\in \mathcal C^{\infty }\left( \mathbb R^{n}\right)$. As a consequence the Hilbert embedding is actually a diffeomorphism from the differential space $\left( V/\Gamma ,\mathcal{C}^{\infty }( V)^{\Gamma }\right)$ to the differential space $\left(\mathcal{X} ,\mathcal{C}^{\infty }(\mathcal{X})\right)$. In this manner $\mathcal{C}^{\infty }(\mathcal{X})$ inherits a Poisson bracket $\{\ ,\ \}$ from $\mathcal{C}^{\infty }( V)^{\Gamma }$.

Our aim is now to construct a non-degenerate closed form $\omega \in \operatorname{Alt}_{\mathcal{C}^{\infty }(\mathcal{X})}^{2}\left(\operatorname{Der}\left(\mathcal{C}^{\infty }(\mathcal{X})\right) ,\mathcal{C}^{\infty }(\mathcal{X})\right)$
from the restriction $\omega _{K}{}_{|\operatorname{Der}\left(\mathcal{C}^{\infty }( V)\right)^{\Gamma }} \in \operatorname{Alt}_{\mathcal{C}^{\infty }( V)^{\Gamma }}^{2}\left(\operatorname{Der}\left(\mathcal{C}^{\infty }( V)\right)^{\Gamma } ,\mathcal{C}^{\infty }( V)^{\Gamma }\right)$ of the Kähler form\footnote{Here Kähler form means a Kähler form in the sense of differential geometry. It should not be confused with a \emph{Kähler differential}.}
$$\omega _{K} \in \operatorname{Alt}_{\mathcal{C}^{\infty }( V)}^{2}\left(\mathcal{C}^{\infty }( V) ,\mathcal{C}^{\infty }( V)\right)=\Omega^2(V)$$
to invariant derivations. By Gerald Schwarz's  \cite[Theorem]{LiftingHomo} there is a short split exact sequence
\begin{align}
0\leftarrow \operatorname{Der}\left(\mathcal{C}^{\infty }( V)^{\Gamma }\right)\underset{\Lambda }{\overset{\pi }{\leftrightarrows }}\operatorname{Der}\left(\mathcal{C}^{\infty }( V)\right)^{\Gamma }\leftarrow \operatorname{Der}_{\Gamma }\left(\mathcal{C}^{\infty }( V)\right)^{\Gamma }\leftarrow 0,
\end{align}
where $\operatorname{Der}_{\Gamma }\left(\mathcal{C}^{\infty }( V)\right)^{\Gamma } =\left\{X\in \operatorname{Der}\left(\mathcal{C}^{\infty }( V)\right)^{\Gamma } \ |\ X_{|\mathcal{C}^{\infty }( V)^{\Gamma }} =0\right\}$. We now want to prove that $\operatorname{Der}_{\Gamma }\left(\mathcal{C}^{\infty }( V)\right)^{\Gamma } =\{0\}$ in order to deduce the non-degeneracy of $\omega _{K}{}_{|\operatorname{Der}\left(\mathcal{C}^{\infty }( V)\right)^{\Gamma }}$ along the lines of Lemma \ref{lem:Chris} and Corollary \ref{cor:nodeg}.

\begin{lemma}
Let $\Gamma \rightarrow \operatorname{U} (V)$ be a unitary representation of the finite group $\Gamma $ on the finite dimensional hermitean vector space $V$. Then any $X\in \operatorname{Der} (\mathcal{C}^{\infty }( V) )$ such that $X_{|\mathcal{C}^{\infty }( V)^{\Gamma }} =0$ has to vanish.
\end{lemma}

\begin{proof} Recall that $\mathcal{C}^{\infty }( V)^{\Gamma }$ separates $\Gamma $-orbits. It is now easy to adapt the argument of Lemma \ref{lem:Chris} to the smooth situation using fact that the image $\boldsymbol{z}(] -\epsilon ,\epsilon [)$ of an integral \ curve $] -\epsilon ,\epsilon [\rightarrow V,\ t\mapsto \boldsymbol{z}( t)$ of a non-constant vector field $X\in \operatorname{Der} (\mathcal{C}^{\infty }( V) )$ cannot be finite.
\end{proof}

We know that the invariant de Rham complex $\left(\operatorname{Alt}_{\mathcal C^{\infty}(V)^{\Gamma}}\left(\operatorname{Der}\left(\mathcal{C}^{\infty }( V)\right)^{\Gamma } ,\mathcal{C}^{\infty }( V)^{\Gamma }\right) ,d_{\operatorname{dR}}\right)$ is a subcomplex of the usual de Rham complex $\left( \Omega ( V) \simeq \operatorname{Alt}_{\mathcal{C}^{\infty }( V)}\left(\operatorname{Der}\left(\mathcal{C}^{\infty }( V)\right) ,\mathcal{C}^{\infty }( V)\right) ,d_{\operatorname{dR}}\right)$
of smooth differential forms on $V$. Arguing along the lines of Section \ref{sec:expl} we can define a closed, non-degenerate form $\omega \in \operatorname{Alt}_{\mathcal{C}^{\infty }( V)^{\Gamma }}^{2}\left(\operatorname{Der}\left(\mathcal{C}^{\infty }( V)^{\Gamma }\right) ,\mathcal{C}^{\infty }( V)^{\Gamma }\right)$ via the formula
$$\omega ( X,Y) =\omega _{K}( \Lambda ( X) ,\Lambda ( Y))$$
for $X,Y\in \operatorname{Der}\left(\mathcal{C}^{\infty }( V)^{\Gamma }\right)$. In fact, for any $\xi ,\eta $ such that $\pi ( \xi ) =X,\ \pi ( \eta ) =Y$ we have $\omega _{K}( \xi ,\eta ) =\omega _{K}( \Lambda ( X) ,\Lambda ( Y))$, so that we are not obliged to use the split $\Lambda $ to define $\omega $.

In conclusion, using the isomorphism $\mathcal{C}^{\infty }(V )^\Gamma\simeq\mathcal{C}^{\infty }( \mathcal X)$, we obtain a symplectic form on the the semialgebraic set $\boldsymbol{u}(V)=\mathcal{X}$, seen as a differential space in the sense of Sikorsky. Our naive de Rham complex $ \left(\operatorname{Alt}_{\mathcal{C}^{\infty }( \mathcal X)}\left(\operatorname{Der}\left(\mathcal{C}^{\infty }(\mathcal{X})\right) ,\mathcal{C}^{\infty }(\mathcal{X})\right) ,\operatorname{d}_{\operatorname{dR}}\right)$ coincides with the one studied recently by Bates, Cushman and Śniatycki in \cite{axioms}. The distinction between the derivations $ \operatorname{Der}\left(\mathcal{C}^{\infty }(\mathcal{X})\right)$ and the vector fields $ \mathfrak{X}^\infty(\mathcal{X})$ that the authors of \cite{axioms} emphasize is actually mute for symplectic orbifolds of cotangent lifted actions of unitary representations of finite groups as there are no strata of codimension one (see Proposition \ref{prop:codim1}).

\section{Conclusion and outlook}\label{sec:out}
We proposed a general framework of how to make sense of a symplectic form for a singular affine Poisson variety and showed that it is not void by exhibiting symplectic forms on categorical quotients of cotangent lifted representations of finite groups. In the case of the double cone we did this in two ways: 1.) using constructive invariant theory and 2.) using the work of Gerald Schwarz on lifting derivations from categorical quotients and orbit spaces. We were also able to elaborate a differential space version of the results to construct symplectic forms on orbit spaces of linear symplectic orbifolds.

The construction of Section \ref{sec:dc} was so simple because for the double cone $\operatorname{Der}( A) /\left( \Omega {_{A|\boldsymbol{k}}}^{\sharp }\right)$ had one generator. In principle, it may happen however that the number of generators of $\operatorname{Der}( A) /\left( \Omega {_{A|\boldsymbol{k}}}^{\sharp }\right)$ is $\geq 2$. If one is attempting to construct the symplectic form the simplest assumption to try is to suppose that the generators of $ \operatorname{Der}_{I}( A)$ not belonging to $ \Omega {_{A|\boldsymbol{k}}}^{\sharp }$ are isotropic. Then the calculations checking closedness of $ \omega $ appear to be straight forward. Yet the catch is that those generators are unique up to $ \Omega {_{A|\boldsymbol{k}}}^{\sharp }$, which in turn is typically not isotropic. It should be said that the concrete data of an affine Poisson algebra are typically bulky, if available at all. A more systematic empirical study must rely on computer implementations to be practically feasible. Of course, a conceptual way to prove the existence of the symplectic form is desirable. It is not unlikely that, using \cite{LiftingHomo, HSScomp} and with appropriate largeness assumptions it is possible to generalize the Marsden-Weinstein Theorem \cite{MW} to singular symplectic reduction. This could be attempted for symplectic reduction in the framework of complex algebraic geometry and in the context of Sjamaar-Lerman reduction \cite{SL}. From this it might be even possible to show directly that the strata of the symplectic quotient are symplectic manifolds by restricting the symplectic form. Another conceptual approach could be to construct a symplectic form on a symplectic singularity \cite{Beauville} from the resolution of the singularity that is employed in Beauville's definition.

%It should be also said that algebraic geometry is not the proper setting for symplectic geometry since Hamiltonian flows do not respect polynomial observables. The appropriate framework for singular symplectic geometry appears to be a Poisson differential space $ \left( X,\mathcal{C}^{\infty }( X) ,\{\ ,\ \}\right)$ in the sense of \cite{FHS}, or variations thereof such as, e.g., \cite{SL,stratKaeh}. This is because symplectic reductions by compact group actions and gauge theoretic moduli spaces are to be described in this language. The notion of the module of differentials $ \mathfrak{D}( X)$ for a differential space has been developed in \cite{Navarro} and the idea of using the naive de Rham complexes of the Lie-Rinehart algebras $ \left(\mathcal{C}^{\infty }( X) ,\mathfrak{D}( X)^{\sharp }\right)$ and $ (\mathcal{C}^{\infty }( X) ,\operatorname{Der}\left(\mathcal{C}^{\infty }( X)\right)$ goes through without difficulty. One faces however the problem that it is not so clear how to gain explicit descriptions of the $ \mathcal{C}^{\infty }( X)$-module $ \operatorname{Der}\left(\mathcal{C}^{\infty }( X)\right)$, since the ideal theory of $ \mathcal{C}^{\infty }\left(\mathbb{R}^{n}\right)$ is more subtle. If one could show that the naive de Rham complex is sufficiently natural it might turn out to be straight forward to show that the singular symplectic form restricts to the symplectic forms on the symplectic strata (compare \cite{SL}).

\bibliographystyle{amsplain}
\bibliography{symplectic.bib}
\end{document}